\documentclass[a4paper,12pt]{amsart}
\usepackage[dvips]{graphicx}
\usepackage[shortlabels]{enumitem}
\usepackage{amsmath,amssymb,color,enumitem,graphics,hyperref,latexsym,pgfplots,tikz,vmargin,yfonts,amscd}
\hypersetup{
unicode=false,   
pdftoolbar=true,  
pdfmenubar=true,  
pdffitwindow=false,  
pdfstartview={FitH}, 
pdftitle={Existence and Properties of Frames of iterations}, 
pdfauthor={Alejandra Aguilera, Carlos Cabrelli, Felipe Negreira, Victoria Paternostro},  
pdfsubject={42C15, 47A15, 94A20}, 
pdfkeywords={Dynamical sampling, frame theory, model spaces}, 
pdfnewwindow=true,  
colorlinks=true,  
linkcolor=black,   
citecolor=black,  
filecolor=magenta,  
urlcolor=cyan   
}

\def\cA{{\mathcal{A}}}
\def\cB{{\mathcal{B}}}
\def\cF{{\mathcal{F}}}
\def\C{{\mathbb{C}}}
\def\D{{\mathbb{D}}}
\def\d{\,{\mathrm{d}}}
\def\HH{{\mathcal{H}}}
\def\cI{{\mathcal{I}}}

\def\KK{{\mathcal{K}}}
\def\LL{{\mathcal{L}}}
\def\GG{{\mathcal{G}}}

\def\cF{{\mathcal{F}}}
\def\T{{\mathbb{T}}}
\def\Z{{\mathbb{Z}}}
\def\N{{\mathbb{N}}}
\def\R{{\mathbb{R}}}
\def\osp{{\overline{\textrm{span}}}}
\def\esssup{{\mathrm{ess}\sup}}

\DeclareRobustCommand{\rchi}{{\mathpalette\irchi\relax}}
\newcommand{\irchi}[2]{\raisebox{\depth}{$#1\chi$}} 

\newcommand{\norm}[1]{{\left\|{#1}\right\|}}
\newcommand{\abs}[1]{{\left|{#1}\right|}}
\newcommand{\scal}[1]{{\langle{#1}\rangle}}

\newtheorem{theorem}{Theorem}[section]
\newtheorem{lemma}[theorem]{Lemma}
\newtheorem{proposition}[theorem]{Proposition}
\newtheorem{corollary}[theorem]{Corollary}

\newtheorem{definition}[theorem]{Definition}
\newtheorem{remark}[theorem]{Remark}

\theoremstyle{definition}
\newtheorem{example}[theorem]{Example}

\begin{document}

\title{Optimal Dynamical Frames}

\let\thefootnote\relax

\author[A. Aguilera, C. Cabrelli, F. Negreira, V. Paternostro]{A. Aguilera, C. Cabrelli, F. Negreira, V. Paternostro}
\address{Departamento de Matem\'atica, Universidad de Buenos Aires, Instituto de Matem\'atica ``Luis Santal\'o'' (IMAS-CONICET-UBA), Buenos Aires, Argentina.}
\email{aaguilera@dm.uba.ar}
\email{carlos.cabrelli@gmail.com}
\email{fnegreira@dm.uba.ar}
\email{vpater@dm.uba.ar}


\maketitle

\begin{abstract}
Let $T$ be a bounded operator acting on an infinite-dimensional Hilbert space. We provide necessary and sufficient conditions on $T$ for the existence of a Parseval frame of iterations.

Based on this result, we give a new proof of a known characterization concerning the existence of (not necessarily Parseval) frames of iterations. As a consequence of our results, we obtain that certain operators lack Parseval frames despite possessing frames of iterations.

Furthermore, we introduce the {\em frame index} of a bounded operator $T$ as the minimal number of vectors required to generate a frame via its iterations. We obtain the exact value of this index for both Parseval frames and general frames of iterations, and we give a constructive method to generate such frames.

Assuming that $T$ satisfies the conditions ensuring that both $T$ and $T^*$ admit frames of iterations, we show how to construct Parseval frames generated by iterations of $T$ and also by iterations of $T^*$. This construction relies on universal models in vector-valued Hardy spaces and the theory of universal dilations. Furthermore, we provide necessary and sufficient conditions under which the frames generated by $T$ and $T^*$ are similar, in terms of the inner function of the associated model in the Hardy-space model associated with $T$.
\end{abstract}

\maketitle

\section{Introduction}

Let $\HH$ be a complex separable infinite-dimensional Hilbert space. We study the class of bounded linear operators $T$ defined on $\HH$ that admit a {\em frame of iterations}; that is, a frame of the form $\{T^n v_i\}_{i \in I,\, n \geqslant 0}$, where $\{v_i\}_{i \in I}$ is a countable collection of vectors in $\HH$ (called {\em generators}).

Understanding the conditions under which an operator $T$ admits such a frame is a central problem, motivated both by its intrinsic operator-theoretic interest and by its relevance to applications where iterative structures naturally arise. One of the main  motivations is the theory of {\em dynamical sampling}. In dynamical sampling, a signal evolving over time is observed through samples taken at fixed spatial locations and at successive time instances. When spatial sampling is sparse, the main challenge is to reconstruct the initial signal from these space-time samples by exploiting temporal evolution.

From a mathematical perspective, the dynamical sampling problem can be reformulated as determining when the orbits of an operator $T$ (the evolution operator) form a frame, that is, when a frame of iterations,  also called a {\em dynamical frame}, exists (see~\cite{ACMT}). The theory of dynamical sampling has developed into an active and rich area of research; see, for instance, \cite{ADK13,AAK19,AT14,ACMT,ACCMP17,AP17,CMPP20,CHS18,CMS21,CHP,DMM21,DMMM2021-1,MMO23,DLW24,Philipp17,Tang17,AK17,ACCP,CHPS24}.

Recall that a collection of vectors $\{x_i\}_{i\in I}\subseteq\HH$ is a
{\em frame} for $\HH$ if there exist constants $A,B>0$ such that
\begin{equation*}
A\norm{x}^2 \leqslant \sum_{i\in I} \abs{\scal{x,x_i}}^2 \leqslant B\norm{x}^2,
\quad \text{for all } x\in\HH.
\end{equation*}
When $A=B=1$ we say that $\{x_i\}_{i\in I}$ is a {\em Parseval frame}.  For each frame $\{x_i\}_{i\in I}$ there exists another frame  called a  dual frame $\{y_i\}_{i\in I}$ such that every $x\in\HH$ can be represented as $x=\sum_{i\in I}\scal{x,x_i}\,y_i$. In general, frames can have many dual frames. This property is of great importance in applications.

In this work, we settle the problem of the existence of frames of iterations, building upon previous results and further developing the theory. In particular, we study the existence of Parseval frames of iterations and prove that the sufficient conditions previously known for their existence \cite{AP17} are also necessary. This establishes a complete characterization of operators that generate Parseval dynamical frames and, consequently, resolves the question of whether there exist operators that do not  admit such frames.

A central contribution of this paper is the study of the {\em minimal number of generators} required to produce a frame of iterations. We formalize this notion through the introduction of the {\em frame index} of a bounded linear operator, which measures the smallest number of vectors whose orbits under the operator form a frame of the space.

In many dynamical sampling scenarios, the goal is to reconstruct the initial state of an evolving system from finitely many measurements obtained by sensors located at fixed spatial positions and recording the signal at successive time instants. In this framework, the sensor positions correspond to the generators of the system. From an applied viewpoint, identifying a {\em minimal set of generators} is of particular importance, since they represent the most efficient configuration of sensors capable of ensuring stable recovery of the evolving signal. Reducing their number not only minimizes costs and addresses practical constraints such as accessibility or physical limitations, but also reveals intrinsic structural properties of the underlying operator. This interplay between efficiency and structure motivates the study of frames of iterations and their minimal generating systems.

In this direction, our main result provides an explicit formula for the minimal number of generators required to produce a Parseval frame of iterations. Then, from this formula, we derive an expression for the general minimal number of generators that give a frame (not necessarily Parseval) of iterations. We also present examples illustrating the difference between the minimal number of generators needed to produce a general frame of iterations and those needed to obtain a Parseval frame of iterations.

Beyond establishing these theoretical conditions, our methods yield a constructive framework for generating frames of iterations for a given operator. We refine these constructions to obtain {\em optimal dynamical frames}, realized by means of a minimal number of generators that are also linearly independent. This is possible once we know the value of this minimal number.

We also show that if a frame of iterations can be generated by a {\em finite} set of vectors, then every minimal set of generators must consist of linearly independent vectors. Moreover, even in the case where the minimal set of generators is {\em infinite}, it can always be chosen to be linearly independent.

Furthermore, we address the question of whether the property of admitting a dynamical frame is preserved by the  adjoint operator. Specifically, assuming that there exists a set $\{v_i\}_{i\in I}\subseteq\HH$ such that $\{T^n v_i\}_{i\in I,\,n\geqslant0}$ forms a frame, does there also exist a set
$\{w_j\}_{j\in J}\subseteq\HH$ such that $\{(T^*)^n w_j\}_{j\in J,\,n\geqslant0}$ is a frame of $\HH$?  If so, what is the relation between these two frames?  Are they similar frames of iterations, in the sense of Definition \ref{def-similar-frames}? We provide affirmative answers to these questions.

Our approach relies on functional models defined in Hardy spaces with multiplicity. By exploiting the rich structure of these spaces, we solve the central problems in this concrete setting and then transfer the results to the abstract framework.

\subsection{Organization of the paper and main results}

We begin by examining the problem of the existence of Parseval frames of iterations. In \cite{AP17}, it was shown that an operator $T$ admits a Parseval frame of iterations if it is a {\em contraction}, i.e. $\norm{T}\leqslant1$, and if its adjoint is {\em strongly stable}, that is  $(T^*)^nv\to0$ as $n\to+\infty$ for all $v\in\HH$. In this paper, we show that these conditions are also necessary, thereby obtaining a complete characterization of operators that generate Parseval frames of iterations.

\begin{theorem}\label{thm:iterationsParseval}
Let $T$ be a bounded linear operator on $\HH$. There exists a set of vectors $\{v_i\}_{i\in I}\subseteq\HH$ such that $\{T^nv_i\}_{i\in I,n\geqslant0}$ is a Parseval frame in $\HH$ if and only if $\norm{T}\leqslant1$ and $(T^*)^nv\to0$ as $n\to+\infty$ for all $v\in\HH$. 
\end{theorem}

 The existence of frames of iterations, not necessarily Parseval, has already been characterized in \cite{CMPP20}.

\begin{theorem}[\cite{CMPP20}]\label{existencia}
Let $T$ be a bounded linear operator on $\HH$. Then there exists a frame of iterations by $T$ if and only if $T$ is similar to a contraction and $(T^*)^nv\to0$ as $n\to+\infty$ for all $v\in\HH$.
\end{theorem}

We give a different proof of Theorem \ref{existencia} using Theorem \ref{thm:iterationsParseval}. Our proof exploits the connection between frames of iterations and {\em model spaces}, i.e. spaces of analytic vector-valued functions in $L^2(\T,\KK)$ whose orthogonal complement in $H^2_\KK$ is shift-invariant. The origin of this idea can likely be traced back to the work of Rota \cite{R}. This connection is crucial for the results developed in the subsequent sections.

In Section \ref{Pre}, we establish the fundamental concepts and definitions required for this paper. Section \ref{proofs} is dedicated to presenting the proofs of the previously mentioned existence results, as well as a methodology for constructing examples of frames of iterations.

Section \ref{index} studies the minimal number of generators needed to produce a dynamical frame for a given operator $T$, called the {\em frame index} of $T$. For Parseval frames, we write $\gamma_p(T)$, and for general frames $\gamma(T)$.  We were able to provide a characterization of the frame index of $T$:
\begin{theorem}\label{thm-index-intro}
Let $T$ be a bounded operator in $\mathcal H$.
\begin{enumerate}[label=(\roman*),ref=(\roman*)]
\item If $T$ is a contraction with strongly stable adjoint, then 
\begin{equation*}
 \gamma_p(T)=\dim\overline{(I_{\HH}-TT^*)(\HH)}.
\end{equation*}
\item 
If 
$T$ is similar to a contraction and has  strongly stable adjoint, then 
\begin{equation*}
\gamma(T)=\min\{\dim\overline{(I_{\HH}-QQ^*)(\HH)} :Q\text{ is a contraction similar to }T\}.
\end{equation*}
\end{enumerate}
Moreover, in both cases, the index is attained by a linearly independent set of generators.
\end{theorem}

This theorem follows from the proofs of Theorem \ref{nikc2} and Proposition  \ref{general index}.

In the case that $T$ has a Parseval frame of iterations, i.e. $\gamma_p(T)$ is positive, clearly $\gamma(T)$ is also positive and $\gamma(T)\leqslant\gamma_p(T)$.
However, these two indices are not always equal for certain operators. For instance, we provide an example of an operator $T$ with $\gamma(T)=1$ and $\gamma_p(T)=\infty$. Also, note that there exist operators with Parseval index zero and positive general index (i.e., operators similar to contractions but with norm greater than one that have strongly stable adjoint). All the results concerning the index of an operator are presented  in Section \ref{index}.

Using Theorem \ref{existencia} and the fact that an operator $T$ is similar to a contraction if and only if its adjoint $T^*$ is also similar to a contraction, we get immediately the following result.

\begin{theorem}\label{coro}
Let $T$ be a bounded linear operator on $\HH$. Assume that $T$ has a frame of iterations. Then $T^*$ has a frame of iterations if and only if $T$ is strongly stable.
\end{theorem}

If a bounded operator $T$
admits a frame of iterations, and so does its adjoint $T^*$, then both operators must be strongly stable. We refer to this joint requirement as {\em joint strong stability}.
 
In Section \ref{full range}, we examine in detail the joint strong stability in the setting of vector-valued Hardy spaces. With the rich structure of the functional models and using the theory of {\em minimal dilations} of Nagy and Foias \cite{NFBK10} we were able to link the joint strong stability with the concept of {\em full range} shift-invariant subspaces of the Hardy spaces. The theory of full range invariant subspaces was developed by Helson in the sixties \cite{H}.
 
This connection enables the construction of frames of iterations via the adjoint of the shift operator in Hardy space models. These are then transferred to the abstract Hilbert space setting. In particular, in Section 6 we develop a method for obtaining optimal frames for contractions satisfying the joint strong stability property. Optimality here refers to minimizing the number of generators.

Finally in Section \ref{estructura}, we address the question of when the frames generated by $T$ and $T^*$ are similar. This part of the work relies on advanced tools, including the Beurling-Lax-Halmos theorem, which characterizes the invariant subspaces under the unilateral shift operator acting on $H^2_\KK$, as well as the concept of inner functions in this context and their properties. We obtain necessary and sufficient conditions under which the frames of iterations by $T$ and $T^*$ are similar frames in terms of properties of the inner functions of the associated functional model in the Hardy space.

\section{Preliminaries}\label{Pre}

Let us recall some notation regarding operators in Hilbert spaces. All the Hilbert spaces considered in this paper are assumed to be separable. Given two Hilbert spaces $\HH$ and $\KK$, $\cB(\mathcal H,\KK)$ denotes the set of linear bounded operators from $\HH$ to $\KK$.  When $\KK=\HH$, we will simply write $\cB(\mathcal H,\HH)=\cB(\HH)$. 
We will say that an operator in $\cB(\mathcal H,\KK)$ is an {\em isomorphism}
if it is invertible.
If $N\subseteq\mathcal H$ is a closed subspace we denote by $P_N$ the orthogonal projection of $\HH$ onto $N$. Finally, if $T\in\mathcal B(\mathcal H)$, a closed subspace $N\subseteq\mathcal H$ is said to be {\em invariant under} $T$ or $T${\em-invariant} if $T(N)\subseteq N$ or, equivalently, $P_NTP_N=TP_N$.

Finally, we recall that $T_1\in\mathcal B(\HH_1)$ and $T_2\in\mathcal B(\HH_2)$ are said to be {\em similar} if there exists an isomorphism $V:\HH_1\to\HH_2$ such that
\begin{equation*}
VT_1V^{-1}=T_2.
\end{equation*}

\subsection{Model spaces}
Let us now give a brief exposition about model spaces, for more details, see e.g. \cite[Lecture $\S$VI]{H}. Let $\T$ denote the one-dimensional torus in $\C$, i.e. 
\begin{equation*}
\T:=\{z\in\C:\abs{z}=1\}
\end{equation*}
endowed with the normalized Lebesgue measure.

Given $\KK$ a Hilbert space, $L^2(\T,\KK)$ is the set of all measurable and square-integrable $\KK$-valued functions, which is a Hilbert space with the inner product
\begin{equation*}
\scal{f,g}:=\int_\T\scal{f(z),g(z)}_\KK\d z\qquad f,g\in L^2(\T,\KK).
\end{equation*}

If $\{e_i\}_{i\in I}$ is an orthonormal basis of $\KK$ and $\rchi:\T\to\T$ with $\rchi(z)=z$, then it can be shown that $\{e_i\rchi^n\}_{i\in I,n\in\Z}$ is an orthonormal basis of $L^2(\T,\KK)$. Using this basis one can write any function $f\in L^2(\T,\KK)$ uniquely as 
\begin{equation*}
f(z)=\sum_{n\in\Z}a_nz^n
\end{equation*}
for almost every $z\in\T$, where $a_n=\sum_{i\in I}\scal{f,e_i\rchi^n}e_i$ for all $n \in \Z$. In fact, $\{a_n\}_{n\in\Z}$ is a square-summable sequence in $\KK$ with $\sum_{n\in\Z}\norm{a_n}_{\KK}^2=\norm{f}^2$. For every integer $n$, $a_n$
denotes the
$n$-th Fourier coefficient of $f$.

The {\em Hardy space}  denoted by $H^2_\KK$, is the closed subspace of $L^2(\T,\KK)$ of functions whose negative Fourier coefficients are zero. That is, the subspace of $L^2(\T,\KK)$ spanned by $\{e_i\rchi^n\}_{i\in I,n\geqslant0}$. The multiplicity of $H^2_{\KK}$ is the $\dim(\KK)$. Observe that the Hilbert space $\KK$ is identified with the subspace of $H^2_{\KK}$ consisting of constant functions. Alternatively, functions in $H^2_\KK$ can be seen as the radial limit of vector-analytic functions in $\D$ such that  their power series coefficients are square-summable. That is, if $f\in H^2_\KK$ then its expansion
\begin{equation*}
f(z)=\sum_{n\geqslant0}a_nz^n
\end{equation*}
defines a holomorphic function in $\D$ with values in $\KK$. Under this identification, we have $a_n=\frac{f^n(0)}{n!}$ for all $n\geqslant0$.

We will denote $U:L^2(\T,\KK)\to L^2(\T,\KK)$ the {\em bilateral shift}, which is given by
\begin{equation*}
Uf(z):=zf(z)
\end{equation*}
for almost every $z\in\T$ and all $f\in H^2_\KK$. Observe that $U$ defines a bounded unitary operator in $L^2(\T,\KK)$ whose adjoint satisfies $U^*f(z)=\frac{f(z)}{z}$ for almost every $z\in\T$ and all $f\in L^2(\T,\KK)$. 

Since $H^2_\KK$ is invariant under $U$, restricting $U$ to ${H^2_\KK}$ gives an isometry $S:H^2_\KK\to H^2_\KK$ usually called the {\em unilateral shift}. Note that $S^*:H^2_\KK\to H^2_\KK$ is given by 
\begin{equation*}
S^*f(z)=P_{H^2_\KK}U^*f(z)=\dfrac{f(z)-f(0)}{z}
\end{equation*}
for almost every $z\in\T$ and all $f\in H^2_\KK$, where $P_{H^2_\KK}$ denotes the projection of $L^2(\T,\KK)$ onto $H^2_\KK$.

A {\em model space} $N$ is a closed subspace of $H^2_\KK$ that is invariant under $S^*$. We will consider the {\em shift compression to} $N$ as the operator $A_N:N\to N$ given by
\begin{equation*}
A_Nf:=P_NSf
\end{equation*}
for all $f\in N$. Note that, since $N$ is $S^*$-invariant, then $(A_N)^*f=S^*f$ for all $f\in N$.

\subsection{Basic frames}
Model spaces and their associated shift compression are key to characterizing frames of iterations in {\em any} Hilbert space. Particularly, we can reduce our analysis to studying sets of the form 
\begin{equation*}
\{A_N^n(P_Ne_i)\}_{i\in I,n\geqslant0},
\end{equation*}
where $N\subseteq H^2_\KK$ is a model space for some Hilbert space $\KK$ and $\{e_i\}_{i\in I}$ is an orthonormal basis of $\KK$. Such sets are indeed always Parseval frames:

\begin{proposition}\label{basicparseval}
Let $\KK$ be a Hilbert space, $N\subseteq H^2_\KK$ a model space and $\{e_i\}_{i\in I}$ an orthonormal basis of $\KK$. Then $\{A_N^n(P_Ne_i)\}_{i\in I,n\geqslant0}$ is a Parseval frame of $N$.
\end{proposition}

The proof follows by straightforward adaptation of {\cite[Proposition $3.6$]{ACCP}}.

Let us now recall the concept of {\em similar  frames of iterations } defined in \cite[$\S$3]{ACCP}. 
\begin{definition}\label{def-similar-frames}
Given two Hilbert spaces $\HH_1,\HH_2$  and $\{T_1^nv_i\}_{i\in I,n\geqslant0}$ and $\{T_2^nw_i\}_{i\in I,n\geqslant0}$ two frames of iterations in $\HH_1$ and $\HH_2$ respectively, we say that they are similar frames if $T_1$ and $T_2$ are similar with isomorphism $V$ and $V(v_i)=w_i$ for all $i\in I$. We further say they are  unitarily equivalent frames of iterations if $V$ is unitary.
\end{definition}

It is not difficult to see that this defines an equivalence relation among frames of iterations. Further, the following proposition states that every frame of iterations in a Hilbert space has a \textit{corresponding frame} in a model space. For a proof see \cite[Theorem 2.1]{CMS}, \cite[Theorem 3.4]{CHP} and \cite[Theorem $3.9$]{ACCP}.

\begin{proposition}\label{basic}
Let $T\in\cB(\HH)$ and $\{v_i\}_{i\in I}\subseteq\HH$. Then,
\begin{enumerate}[label=(\roman*),ref=(\roman*)]
\item If $\{T^nv_i\}_{i\in I,n\geqslant0}$ is a frame, then it is similar to a frame of the form\begin{equation*}
\{A_N^n(P_Ne_i)\}_{i\in I,n\geqslant0},
\end{equation*}
where $N\subseteq H^2_{\ell^2(I)}$ is a model space and $\{e_i\}_{i\in I}$ is the canonical basis of $\ell^2(I)$.

\item If $\{T^nv_i\}_{i\in I,n\geqslant0}$ is a Parseval frame, then it is unitarily equivalent to a frame of the form
\begin{equation*}
\{A_N^n(P_Ne_i)\}_{i\in I,n\geqslant0},
\end{equation*}
where $N\subseteq H^2_{\ell^2(I)}$ is a model space and $\{e_i\}_{i\in I}$ is the canonical basis of $\ell^2(I)$.
\end{enumerate}
\end{proposition}

For this reason, frames of the form $\{A_N^n(P_Ne_i)\}_{i\in I,n\geqslant0}$ have been coined as {\em basic frames}.

\section{Existence of frames of iterations }\label{proofs}
In this section, we provide necessary and sufficient conditions for operators to admit frames of iterations.
For a given Hilbert space $\HH$ we denote by $\cF_\HH$ the subset of operators in $\cB(\HH)$ that admit a frame of iterations, and by $\cF^{p}_\HH$ the subset that admit a Parseval frame of iterations. 

Regarding similarities and unitary equivalences, we have the first easy but very useful result:

\begin{proposition}\label{similares}
Let $T_1\in\cB(\HH_1)$ and $T_2\in\cB(\HH_2)$. Then,
\begin{enumerate}[label=(\roman*),ref=(\roman*)]
\item If $T_1$ and $T_2$ are similar and $T_1\in\cF_{\HH_1}$, then $T_2\in\cF_{\HH_2}$.

\item If $T_1$ and $T_2$ are  unitarily equivalent and $T_1\in\cF^p_{\HH_1}$, then $T_2\in\cF^p_{\HH_2}$.
\end{enumerate}
\end{proposition}
\begin{proof}
We only prove $(i)$ as $(ii)$ is analogous. Suppose then that $T_2=VT_1V^{-1}$ for some isomorphism $V:\HH_1\to\HH_2$ and that $\{T_1^nv_i\}_{i\in I,n\geqslant0}$ is a frame in $\HH_1$. Then define, for each $i\in I$, $w_i:=Vv_i$, and note that
\begin{equation*}
T_2^nw_i=V(T_1^nv_i)
\end{equation*}
holds for all $i\in I$ and $n\geqslant0$. Since frames are preserved by isomorphisms, we get that $\{T_2^nw_i\}_{i\in I,n\geqslant0}$ is a frame in $\HH_2$.
\end{proof}

Combining Propositions \ref{similares}, \ref{basicparseval}, and \ref{basic}, we obtain the following result:

\begin{theorem}\label{cms}
Let $T\in\cB(\HH)$. Then,
\begin{enumerate}[label=(\roman*),ref=(\roman*)]
\item $T\in\cF_{\HH}$ if and only if there exists a model space $N\subseteq H^2_\KK$ for some Hilbert space $\KK$ such that $T$ is  similar to $A_N.$

\item $T\in\cF^{p}_\HH$ if and only if there exists a model space $N\subseteq H^2_\KK$ for some Hilbert space $\KK$ such that $T$ is  unitarily equivalent to $A_N.$
\end{enumerate}
\end{theorem}
\begin{proof}
We prove $(ii)$ since the proof of $(i)$ is  analogous. 

First, if $\{T^nv_i\}_{i\in I,n\geqslant0}$ is a Parseval frame in $\HH$ for some vectors $\{v_i\}_{i\in I}\subseteq\HH$, then by $(ii)$ of Proposition \ref{basic} such frame must be unitarily equivalent to a basic frame $\{A_N^n(P_Ne_i)\}_{i\in I,n\geqslant0}$ where $N\subseteq H^2_{\ell^2(I)}$ is a model space. In particular, this means that $T=VA_NV^{-1}$ for some unitary isomorphism $V:N\to\HH$.

Conversely, suppose $T=VA_NV^{-1}$ for some unitary map $V:N\to\HH$ and a model space $N\subseteq H^2_\KK$. Then, as Proposition \ref{basicparseval} tells us that $A_N\in\cF^p_N$, we can use Proposition \ref{similares} to conclude that $T\in\cF^p_\HH$. 
\end{proof}

When trying to exploit the similarity or unitary equivalence given by Theorem \ref{cms}, there are two properties of $A_N$ that stand out: $A_N$ is a contraction and its adjoint is strongly stable. Recall that given a Hilbert space $\HH$, we say that a bounded operator $T\in\cB(\HH)$ is a {\em contraction} if $\norm{T}\leqslant1$. We say that $T$ is {\em strongly stable} if $\lim_{n\to+\infty}T^nv=0$ holds for all $v\in\HH$. Once these two properties are verified for shift compressions, one of the directions of Theorem \ref{thm:iterationsParseval} can be obtained by showing that the properties are preserved by unitary equivalences.

\begin{proposition}\label{prop:parseval-ida}
Let $T\in\mathcal B(\mathcal H)$. If $T\in\cF^{p}_\HH$, then $T$ is a contraction and $T^*$ is strongly stable.
\end{proposition}
\begin{proof}
 From $(ii)$ of Theorem \ref{cms}, we know that there exists a Hilbert space $\KK$ and a model space $N\subseteq H^2_\KK$ such that $T=VA_NV^{-1}=VA_NV^*$ for a unitary map $V:N\to\HH$. 

First, since $V$ is unitary, $\norm{T}=\norm{A_N}=\norm{P_NS|_N}\leqslant\norm{P_N}\norm{S} = 1$ and hence, $T$ is a contraction. 

Second, as $(A_N)^*=S^*|_N$, then for $f=\sum\limits_{n=0}^{+\infty}a_n\rchi^n$ we have
\begin{equation*}
\norm{(S^*)^nf}^2=\sum\limits_{m\geqslant n}\|a_m\|_\KK^2
\end{equation*}
which goes to zero when $n\to +\infty$ as it is the tail of a convergent sum. The strong stability of $T^*$ then follows from the identity $(T^*)^n=V(S^*|_N)^nV^*$ for all $n\geqslant0$.
\end{proof}

The converse of Proposition \ref{prop:parseval-ida} was obtained in \cite{AP17}. We give a proof for completeness.
By Theorem \ref{cms}, it is enough to show that contractions with strongly stable adjoints are unitarily equivalent to shift compressions in a model space. The following result uses Rota functional model \cite{R}, later refined by De Branges and Rovnyak in \cite{BR64}. Since its proof will be important in what follows, we include it below. 
\begin{theorem}[{\cite[p. 18]{N}}]\label{niko}
Let $T\in\mathcal B(\mathcal H)$. If $T$ is a contraction and $T^*$ is strongly stable, then there exists a Hilbert space $\KK$ and a model space $N\subseteq H^2_{\KK}$ such that $T$ is unitarily equivalent to $A_N$ .
\end{theorem}
\begin{proof}
As $\norm{T}\leqslant1$ we have that $I_\HH-TT^*$ is a self-adjoint positive semidefinite operator. Set $D_{T^*}:=(I_\HH-TT^*)^{1/2}$, $\KK=\overline{D_{T^*}(\HH)}$ and define $L:\HH\to H^2_\KK$ as
\begin{equation*}
Lv:=\sum_{n\geqslant0}D_{T^*}(T^*)^nv\:\rchi^n.
\end{equation*}
We claim that $L$ is an isometry:
\begin{align*}
\norm{Lv}^2=\sum_{n\geqslant0}\norm{D_{T^*}(T^*)^nv}_\HH^2&=\sum_{n\geqslant0}\scal{D_{T^*}(T^*)^nv,D_{T^*}(T^*)^nv}
\\
&=\sum_{n\geqslant0}\scal{(I_\HH-TT^*)(T^*)^nv,(T^*)^nv}
\\
&=\sum_{n\geqslant0}\left(\norm{(T^*)^nv}^2-\norm{(T^*)^{n+1}v}^2\right)=\norm{v}^2,
\end{align*}
where the last equality follows from the fact that $(T^*)^nv\to0$ as $n\to+\infty$.

Also, note that $S^*L=LT^*$. Indeed, for every $v\in\HH$
\begin{equation*}
S^*Lv=\sum_{n\geqslant0}D_{T^*}(T^*)^{n+1}v\:\rchi^n=\sum_{n\geqslant0}D_{T^*}(T^*)^n(T^*v)\:\rchi^n=LT^*v.
\end{equation*}
Therefore, the image of $L$ is an $S^*$-invariant subspace of $H^2_\KK$ and we call it $N$. Moreover if $V:N\to\HH$ is the inverse of $L:\HH\to N$, we have that $T^*=V S^*|_N V^{-1}$ and that $V$ is a unitary map. Taking adjoints, we get $T=VA_NV^{-1}$.
\end{proof}

For a contraction $T$, the operators
\[
D_T = (I_{\mathcal{H}} - T^*T)^{1/2}
\quad \text{and} \quad
D_{T^*} = (I_{\mathcal{H}} - TT^*)^{1/2}
\]
are called the \emph{defect operators} of $T$ and $T^*$ respectively, the closures of their ranges
\[
\mathcal{D}_T = \overline{D_T(\mathcal{H})}
\quad \text{and} \quad
\mathcal{D}_{T^*} = \overline{D_{T^*}(\mathcal{H})}
\]
are called the \emph{defect spaces}, and the dimensions $\dim \mathcal{D}_T$ and $\dim \mathcal{D}_{T^*}$ are the \emph{defect indices}; see \cite{NFBK10}.

As a consequence of Theorem \ref{niko}, we derive the remaining part of Theorem \ref{thm:iterationsParseval}. 

\begin{proposition}\label{prop:parseval-vuelta}
Let $T\in\mathcal B(\mathcal H)$. If $T$ is a contraction and $T^*$ is strongly stable, then $T\in\cF^{p}_\HH$.
\end{proposition}
\begin{proof}
If $T$ is a contraction and $T^*$ is strongly stable then Theorem \ref{niko} tells us that $T$ is unitarily equivalent to a shift compression $A_N$ for some model space $N\subseteq H^2_\KK$. Then, by $(ii)$ of Theorem \ref{cms} we have that $T\in\cF^{p}_\HH$.
\end{proof}

By combining Propositions \ref{prop:parseval-ida} and \ref{prop:parseval-vuelta}, we obtain Theorem \ref{thm:iterationsParseval}. With this result, we can provide another proof for Theorem \ref{existencia} which differs from the one given in \cite{CMPP20}.

\begin{proof} [Proof of Theorem \ref{existencia}]
If $T\in\cF_\HH$, by $(i)$ of Theorem \ref{cms}, we have that there exists a Hilbert space $\KK$ and a model space $N\subseteq H^2_\KK$ such that $T=VA_NV^{-1}$ where $V:N\to\HH$ is an isomorphism. As $A_N\in\cF^p_N$ by Proposition \ref{basicparseval}, applying Proposition \ref{prop:parseval-ida} to $A_N$ yields that $A_N$ is a contraction with strongly stable adjoint. But then $T$ is similar to a contraction and its adjoint is strongly stable.

Assume now that $T$ is similar to a contraction $L:\HH_1\to\HH_1$ and $T^*$ is strongly stable. Thus $T=C^{-1}LC$ for an isomorphism $C:\HH\to\HH_1$. Then, we have $(L^{*})^n=(C^*)^{-1}(T^*)^nC^*$ for all $n\geqslant0$, and thus $L^{*}$ is also strongly stable. Therefore, applying Proposition \ref{prop:parseval-vuelta} to $L$, we have that $L\in\cF^{p}_{\HH_1}$. Finally, since $T$ and $L$ are similar, we can use Proposition \ref{similares} to conclude that $T\in\cF_\HH$.
\end{proof}

\subsection{Examples of frames of iterations}
As a consequence of Theorem \ref{existencia} we have the following results about the existence of frame of iterations for bounded operators: first, as noted in \cite{AP17}, if an operator $T:\HH\to\HH$ is unitary, then it cannot have frames of iterations as its adjoint is not strongly stable. Second, extending a result from \cite{CH23}, if an operator $T:\HH\to\HH$ is hypercyclic (i.e. there exists a vector $v\in\HH$ for which $\{T^nv\}_{n\geqslant0}$ is dense in $\HH$) then it cannot have frames of iterations since contractions cannot be hypercyclic and this property is preserved by similarity.

In \cite{AP17} the authors construct a Parseval frame for any operator satisfying the hypothesis of Theorem \ref{thm:iterationsParseval}.
It is not difficult to modify that example to construct a bigger family of frames of iterations.
We recall first that a collection of vectors $\{x_i\}_{i\in I}\subseteq\HH$ is a
{\em Riesz basis} of $\HH$ if it is complete and there exist constants $A,B>0$ such that
\begin{equation*}
A\sum_{i\in I}|c_i|^2 \leqslant \|\sum_{i\in I} c_ix_i\|^2 \leqslant B\sum_{i\in I}|c_i|^2,
\end{equation*}
for all sequence of complex numbers $\{c_i\}_{i\in I}$ with finite support.

\begin{proposition}\label{examples}
Assume that $T\in\cB(\HH)$ is a contraction and $T^*$ is strongly stable. Let  $\KK=\overline{D_{T^*}(\HH)}$. If $\GG=\{g_i\}_{i\in I}$ is a frame of $\KK$, then $\{T^nD_{T^*} g_i\}_{i\in I, n\geqslant0}$ is a frame of $\HH$. 
Furthermore, if $\GG$ is a Riesz basis or an orthonormal basis of $\KK$ then $\{T^n D_{T^*} g_i\}_{i\in I,n\geqslant0}$ also forms a Riesz basis or an orthonormal basis of $\HH$.
\end{proposition}
\begin{proof}
Let $n\geqslant 0$ and $f\in\HH$, then using that $D_{T^*}$ is self-adjoint, we have
\begin{equation*}
\sum_{i\in I}\abs{\scal{ f,T^nD_{T^*}g_i}}^2=\sum_{i\in I}\abs{\scal{ D_{T^*}(T^*)^nf, g_i}}^2.
\end{equation*}
If $\GG$ is a frame of $\KK$ with bounds $A,B>0$, then
\begin{equation}\label{DD}
A\norm{D_{T^*}(T^*)^nf}^2\leqslant\sum_{i\in I}\abs{\scal{ D_{T^*}(T^*)^nf,g_i}}^2\leqslant B\norm{D_{T^*}(T^*)^nf}^2.
\end{equation}
Now, for every $n\geqslant 0$, as we already computed  in the proof of Theorem \ref{niko},
\begin{align*}
\norm{D_{T^*}(T^*)^nf}^2 
&= 
\norm{(T^*)^nf}^2 -\norm{(T^*)^{n+1}f}^2. 
\end{align*}
Then, summing in $n\geqslant 0$ in \eqref{DD} and using the strong stability of $T^*$, we get that
\begin{equation*}
A\norm{f}^2\leqslant\sum_{n\geqslant 0 }\sum_{i\in I}\abs{\scal{ f,T^nD_{T^*}g_i}}^2\leqslant B\norm{f}^2,
\end{equation*}
which shows that $\{T^nD_{T^*} g_i\}_{i\in I,n\geqslant 0}$ is a frame of $\HH$.
The  arguments for the Riesz basis and orthonormal basis cases are analogous.
\end{proof}

Note that in Proposition \ref{examples} if $\GG$ is a Parseval frame of $\KK$, then the frame of iterations of $T$ generated by $\{D_{T^*}g_{i}\}_{i\in I}$ is also Parseval. 

We will now give another construction of frames of iterations.
Let us first prove the follo\-wing lemma. Recall that a pure isometry $L$ is an isometry whose unitary part $\bigcap_{n\geqslant 0}L^n(\HH)=\{0\}.$ 

\begin{lemma}\label{lem:wandering-frame}
Let $L\in\cB(\HH)$ be a pure isometry and $W=\HH\ominus L(\HH)$ the associated wandering subspace. If $\GG=\{g_i\}_{i\in I}$ is a frame of $W$, then $\{L^n g_i\}_{i\in I,n\geqslant0}$ is a frame of $\HH$.
Furthermore, if $\GG$ is a Riesz basis or an orthonormal basis of $W$ then $\{L^n g_i\}_{i\in I,n\geqslant0}$ also forms a Riesz basis or an orthonormal basis of $\HH$.
\end{lemma}
\begin{proof}
 Since  $L$ is an isometry, it follows that $LL^* = P_{L(\HH)}$. Then, $I_{\mathcal{H}} - LL^* = I_{\HH} - P_{L(\HH)} = P_W$. Thus, the defect space of $L^*$ is equal to $W$. Moreover, since $L$ is a pure isometry, then $L^*$ is strongly stable. See for example \cite[Section 1.3]{RR85}.

Therefore, by Proposition \ref{examples},
if $\{g_i\}_{i\in I}$ is a frame of $W$, then $\{L^ng_i\}_{i\in I, n\geqslant 0}$ is a frame of $\HH$ because $D_{L^*}g_i = (I_{\HH} - LL^*)^{1/2} g_i = P_W g_i= g_i,$ for every $i\in I$.
The same holds for the Riesz and orthonormal basis cases.
\end{proof}

Now, we can apply this result to construct frames of any model subspace in $H^2_\KK$, in particular of $H^2_\KK$ itself. 

\begin{proposition}\label{propV}
Let $\KK$ be a Hilbert space and $N\subseteq H^2_\KK$ a model space. 
If $\{g_i\}_{i\in I}$ is a frame of $\KK$, then $\{A^n_N(P_N g_i)\}_{i\in I, n\geqslant0}$ is a frame of $N$.

In particular, $\{S^n g_i\}_{i\in I, n\geqslant0}$ is a frame of $H^2_\KK$.
\end{proposition}
\begin{proof}
Since $\KK$ is the wandering subspace associated to $S$ in $H^2_\KK$ and $S$ is a pure isometry, we can apply Lemma \ref{lem:wandering-frame} to conclude that $\{S^ng_i\}_{i\in I, n\geqslant0}$, is a frame of $H^2_\KK$. Then, the orthogonal projection onto $N$, $\{P_N(S^ng_i)\}_{i\in I, n\geqslant0}$, is a frame of $N$. 
Thus, the result follows by noticing that for every $n\geqslant0$ and $i\in I$, $P_N(S^ng_i) = (P_NS)^n P_N g_i = A^n_N(P_Ng_i)$.
\end{proof}

Combining this proposition with Theorem \ref{cms} we can explicitly construct frames of iterations for any $T\in\cF_\HH$.

\begin{corollary}
Let $T\in\cF_\HH$ and $N\subseteq H^2_\KK$ be a model space such that $T$ is similar to $A_N$ through the isomorphism $C:N\rightarrow\HH$. Then, $\{T^n (C(P_N g_i))\}_{i\in I, n\geqslant0}$ is a frame of $\HH$, where $\{g_i\}_{i\in I}$ is an arbitrary frame of 
$\KK.$
\end{corollary}

\section{Frame Index of a Bounded Operator}\label{index}
\subsection{Frame indexes}
We now turn our attention to the problem of determining the minimum number of generating vectors required to obtain a frame of iterations by a given operator. 

Recall that, for a Hilbert space $\HH$ and $T\in\cB(\HH)$, we say that a collection of non-zero vectors $\{v_i\}_{i\in I}\subseteq\HH$ is a set of {\em generators} for $T$ if $\{T^nv_i\}_{i\in I,n\geqslant0}$ is a frame of $\HH$. We will further say that $\{v_i\}_{i\in I}$ is a set of {\em Parseval generators} for $T$ if $\{T^nv_i\}_{i\in I,n\geqslant0}$ is a Parseval frame of $\HH$.
\begin{definition}
The  frame index of $T$, denoted $\gamma(T)$, is defined as:
\begin{equation*}
\gamma(T)=\min\{ d\in\N :\text{there exists a set of $d$ generators for $T$}\},
\end{equation*}
if $T$ possesses a frame of iterations that admits a finite generating set. We say that $\gamma(T)=0$ if $T$ does not possess a frame of iterations, and $\gamma(T)=+\infty$ if $T$ possesses a frame of iterations, but none admit a finite generating set.

The {\em Parseval frame index}, $\gamma_p(T)$ has the same definition, but requiring that the frame of iterations is a Parseval frame.
\end{definition}

For the Parseval index we are able to obtain a precise formula and along the way show that such index can be achieved using a set of generators which are linearly independent. The next result shows item (i) of Theorem \ref{thm-index-intro}.

\begin{theorem}\label{nikc2}
Let $T\in\cB(\HH)$. If $T$ has a Parseval frame of iterations, then
\begin{equation*}
\gamma_p(T)=\dim\overline{(I_\HH-TT^*)(\HH)}.
\end{equation*}
Moreover, the index $\gamma_p(T)$ is attained with linearly independent generators.

\end{theorem}
\begin{proof}
Recall that  $D_{T^*}=(I_{\HH}-TT^*)^{1/2}$.
As an initial step, we will show that if $\{w_i\}_{i\in I}$ is a set of Parseval generators, then $\#I\geqslant\dim\overline{(I_\HH-TT^*)(\HH)}$. To see this, first note  that if $\{T^nw_i\}_{i\in I,n\geqslant0}$ is a Parseval frame, by Proposition \ref{basic}, we know that there exists a model space $N\subseteq H^2_{\ell^2(I)}$ and a unitary operator $V:N\to\HH$ such that $T=VA_NV^{-1}=VA_NV^*$. Then, on one hand
\begin{equation*}
D_{T^*}^2=I_\HH-TT^*=VV^*-VA_NV^*VS^*|_NV^*=V(I_N-A_NS^*|_N)V^*.
\end{equation*}
On the other hand, if $f\in N$ is written as $f=\sum_{n\geqslant0}a_n\rchi^n$ with $\{a_n\}_{n\geqslant0}\subseteq\KK$ then
\begin{align*}
(I_N-A_NS^*|_N)f=P_N(I_{H^2_{\KK}}-SS^*)f&=P_N\left(\sum_{n\geqslant0}a_n\rchi^n-\sum_{n\geqslant1}a_n\rchi^n\right)
\\
&=P_N(a_0)=\sum_{i\in I}\scal{a_0,e_i}P_N(e_i),
\end{align*}
where $\{e_i\}_{i\in I}$ is the canonical basis of $\ell^2(I)$. Thus, the image of $I_N-A_NS^*|_N$ is spanned by $\{P_N(e_i)\}_{i\in I}$, which is a set with at most $\#I$ elements. Altogether,
\begin{equation*}
\#I\geqslant\dim\overline{(I_N-A_NS^*|_{N})(N)}=\dim\overline{(I_\HH-TT^*)(\HH)}=\dim\overline{D_{T^*}^2(\HH)}. 
\end{equation*}
Now, since $T\in\cF^{p}_{\HH}$, by Proposition \ref{prop:parseval-ida}, $T$ must be a contraction and $T^*$ strongly stable. Then, Theorem \ref{niko} tells us that $T$ must be unitarily equivalent to a shift compression $A_{N_o}$ where $N_o\subseteq H^2_{\KK_o}$ is a model space with $\KK_o:=\overline{D_{T^*}(\HH)}$. Mapping a basic frame $\{A_{N_o}^n(P_{N_o}e_i)\}_{i\in I,n\geqslant0}$ of $N_o$ through the unitary operator that gives the equivalence as in Proposition \ref{similares}, we can construct a Parseval frame in $\HH$ of the form $\{T^nv_i\}_{i\in I,n\geqslant0}$ where $\#I=\dim\KK_o=\dim\overline{D_{T^*}(\HH)}$. Since $D_{T^*}$ is self adjoint, then $\overline{D_{T^*}(\HH)}=\overline{D_{T^*}^2(\HH)}$, which implies
\begin{equation*}
\#I=\dim \overline{D_{T^*}^2(\HH)}=\dim\overline{(I_\HH-TT^*)(\HH)}.
\end{equation*}

Let us end by showing that the set of generators $\{v_i\}_{i\in I}$ we obtained is actually linearly independent.To this end, it suffices to show that the set of vectors $\{P_{N_o}e_i\}_{i\in I}$ are linearly independent when $\{e_i\}_{i\in I}$ is an orthonormal basis of $\KK_o$. Indeed, as $\{v_i\}_{i\in I}$ are the image under an isomorphism (in fact a unitary map) of $\{P_{N_o}e_i\}_{i\in I}$ then, linear independence of $\{P_{N_o}e_i\}_{i\in I}$ implies that of $\{v_i\}_{i\in I}$.

The result  follows if we show that $\KK_o\cap N^\perp_o=\{0\}$ since, in such case, we would have that $P_{N_o}|_{\KK_o}$ is injective and thus $\{P_{N_o}e_i\}_{i\in I}$ is linearly independent.

To see this, recall that $N_o$ is the image of the map $L:\HH\to H^2_{\KK_o}$ defined as
\begin{equation*}
Lv=\sum_{n\geqslant0}D_{T^*}(T^*)^nv\:\rchi^n,\quad v\in\HH.
\end{equation*}
 Take now $k\in\KK_o$ such that $k\perp N_o$. We then have that
\begin{equation*}
\textstyle\scal{k,\sum_{n\geqslant0}D_{T^*}(T^*)^nv\:\rchi^n}_{H^2_{\KK_o}}=0
\end{equation*}
for all $v\in\HH$. But $\textstyle\scal{k,\sum_{n\geqslant0}D_{T^*}(T^*)^nv\:\rchi^n}_{H^2_{\KK_o}}=\scal{k,D_{T^*}v}_\HH$, thus
\begin{equation*}
\scal{k,D_{T^*}v}_\HH=0
\end{equation*}
for all $v\in\HH$, or equivalently $k\in(D_{T^*}(\HH))^\perp=\KK_o^\perp$. Hence $k=0$, and we are finished.
\end{proof}

Let us remark that the inequality $\# I\geqslant\dim\overline{(I_\HH-TT^*)(\HH)}$, where $I$ is the index set of a generating set of vectors, can also be deduced from the arguments given in \cite{CMPP20} which again differ from the ones we use here.

The first attempt to determine the value of the frame index for a general operator appeared in \cite{CMS}. In that work, the authors focused on studying the frame index for a model operator, that is, the compression of the shift acting on a model space. Since every operator that admits a frame of iterations is similar to a model operator, and since the index is preserved under similarity, it suffices to study this particular case. Using deep results from complex analysis and operator theory, such as the Corona theorem and the Gelfand transform in the algebra \(H^{\infty}\), they obtained a lower bound for the index, which they conjectured to be sharp.

We use a different approach and obtain the following result. Recall that when a bounded operator \(T\) admits a frame of iterations, i.e. \(\gamma(T) > 0\), then, by Theorem~\ref{existencia}, \(T\) must be similar to a contraction \(Q \in \cB(\KK)\), and \(T^*\) must be strongly stable. It is straightforward to see that \(T^*\) is similar to \(Q^*\) and that strong stability is preserved under similarity. Consequently, \(Q\) is a contraction and \(Q^*\) is strongly stable, which, by Theorem~\ref{thm:iterationsParseval}, implies that \(\gamma_p(Q) > 0\). On the other hand, as these indexes are preserved by similarity, we have
\begin{equation*}
0<\gamma(T)=\gamma(Q)\leqslant\gamma_p(Q)=\dim\overline{(I_{\KK}-QQ^*)(\KK)}.
\end{equation*}
Moreover, this holds for any contraction $Q$ similar to $T$.

Further, if we take a frame of iterations $\{T^nv_i\}_{i\in I,n\geqslant0}$ with $\#I=\gamma(T)$ generators we know from $(i)$ of Proposition \ref{basic} that there exists a model space $N\subseteq H^2_{\ell^2(I)}$ such that $A_N$ is similar to $T$. Now, as $\{A_N^n(P_Ne_i)\}_{i\in I,n\geqslant0}$ is a Parseval frame of $N$, cf. Proposition \ref{basicparseval}, we have that
\begin{equation*}
\gamma(T)=\#I\geqslant\gamma_p(A_N).
\end{equation*}

We also note that it is enough to consider contractions acting in $\HH$ because the index is preserved by similarity.
These arguments prove the following proposition.
\begin{proposition}\label{general index}
Given a bounded operator $T\in\cB(\HH)$ we have: 
\begin{align*}
\gamma(T) &=\min\;\{\gamma_p(Q) :Q\in\cF^{p}_\HH, Q\text{ similar to } T\}\\
&=\min\;\{\dim\overline{(I-QQ^*)(\HH)} :Q\in\cF^{p}_\HH, Q
\text{ similar to } T\}.
\end{align*}
Here we take $\min\{\emptyset\}=0$.
\end{proposition}
Let us now illustrate, by means of examples, how these indexes $\gamma(T)$ and $\gamma_p(T)$ coincide or differ.

\begin{example}\label{example}
We begin with a specific case where $0<\gamma(T)<\gamma_p(T)=+\infty$. To that end, we will consider the operator and frames introduced in \cite{ACMT}. Let $\{\lambda_k\}_{k\in\N}$ be a sequence in the complex unit disc $\mathbb{D}$ satisfying $\sum_{k\in\N}(1-\abs{\lambda_k}^2)<\infty$ and the {\em Carleson condition}
\begin{equation*}
\inf_{k\in\N}\prod_{j\neq k}\frac{\abs{\lambda_k-\lambda_j}}{1-\overline{\lambda_k}\lambda_j}>0
\end{equation*}
and define $T:\ell^2(\N)\to\ell^2(\N)$ as $Te_k=\lambda_ke_k$ for all $k\in\N$, where $\{e_k\}_{k\in\N}$ is the canonical basis of $\ell^2(\N)$. It is proved in \cite[Theorem 3.6]{ACMT} that there exists a vector $f\in\ell^2(\N)$ such that $\{T^nf\}_{n\geqslant0}$ is a frame of $\ell^2(\N)$. Thus, $\gamma (T)=1$. 

On the other hand, it is easily seen that $\norm{T}=\sup\{\abs{\lambda_k}:k\in\N\}$. Moreover, since $\sum_{k\in\N}(1-\abs{\lambda_k}^2)<\infty$, it holds that $|\lambda_k|\to 1$ as $k\to+\infty$ and thus $\norm{T}=1$. This means that $T$ is a contraction and, as $T\in\mathcal F_{\ell^2(\N)}$, we have $T^*$ is strongly stable. Therefore, by Proposition \ref{prop:parseval-vuelta}, $T\in\mathcal{F}^p_{\ell^2(\N)}$. Finally, Theorem \ref{nikc2} gives us that $\gamma_p(T)=\dim\overline{(I_{\ell^2(\N)}-TT^*)(\ell^2(\N))}$ and a direct calculation shows that $\dim\overline{(I_{\ell^2(\N)}-TT^*)(\ell^2(\N))}=+\infty$.

Further, we can slightly adjust the previous construction to obtain, for 
each natural number $N>1$, a contraction $T_N:\ell^2(\N)\to\ell^2(\N)$ 
satisfying $\gamma(T_N)=N$ and $\gamma_p(T_N)=+\infty$. Indeed, using the 
same Carleson sequence $\{\lambda_k\}_{k\in\N}$, define
\[
    T_N
    :=\sum_{j=1}^{N-1}\lambda_1 P_j
     +\sum_{j=N}^{\infty}\lambda_{\,j-(N-1)} P_j,
\]
where $P_j$ denotes the projection onto the $j$-th coordinate. Then 
$\dim\ker(T_N-\lambda_1 I)=N$, and hence $\gamma(T_N)\ge N$: if one tries 
to generate a frame by iterating fewer than $N$ vectors under $T_N$, there 
always exists a nonzero vector in $\ker(T_N-\lambda_1 I)$ orthogonal to 
all such iterates.

To show that $\gamma(T_N)$ is exactly $N$, set the operators $T_{N,1}:=\sum_{j=1}^{N-1}\lambda_1P_j$ and $T_{N,2}:=\sum_{j=N}^\infty\lambda_{j-(N-1)}P_j$. Let $e_1,\dots,e_{N-1}$ be the first $N-1$ vectors of the canonical basis. Since $\abs{\lambda_1}<1$, then $C_1=\sum_{n=0}^\infty\abs{\lambda_1}^{2n}$ is finite and so, for $v\in\text{span}\{e_1,\dots,e_{N-1}\}$ we have that 

\begin{align*}
\sum_{n\geqslant0}\sum_{i=1}^{N-1}\abs{\scal{v,(T_{N,1})^ne_i}}^2&=\sum_{n\geqslant0}\sum_{i=1}^{N-1}\abs{\scal{v,(\lambda_1)^ne_i}}^2=\sum_{n\geqslant0}\abs{\lambda_1}^{2n}\sum_{i=1}^{N-1}\abs{\scal{v,e_i}}^2
\\
&=\sum_{n\geqslant0}\abs{\lambda_1}^{2n}\norm{v}^2=C_1\norm{v}^2.
\end{align*}

Thus, $\{(T_{N,1})^ne_i\}_{i=1,\dots,N-1,n\geqslant0}$ is a frame of $\text{span}\{e_1,\dots,e_{N-1}\}$.

Next, for $T_{N,2}$, note that, up to ignoring the first $N-1$ coordinates, we are in the same situation as the initial $T$. Hence, again by \cite{ACMT}, we can take $f\in\ell^2(\N)\ominus\text{span}\{e_1,\dots,e_{N-1}\}$ so that $\{(T_{N,2})^nf\}_{n\geqslant0}$ is a frame of $\ell^2(\N)\ominus\text{span}\{e_1,\dots,e_{N-1}\}$. 

Altogether, $\{(T_{N,1})^ne_i\}_{i=1,\dots,N-1,n\geqslant0}\cup\{(T_{N,2})^nf\}_{n\geqslant0}$ forms a frame of $\ell^2(\N)$. Noting that $(T_{N,1})^ne_i=(T_N)^ne_i$ for all $i=1,\dots,N-1$, $n\geqslant0$ and also $(T_{N,2})^nf=(T_N)^nf$ for all $n\geqslant0$, we have that $\{e_1,\dots,e_{N-1},f\}$ is a generator set for $T_N$ and thus $\gamma(T_N)=N$. Also, we still have $\gamma_p(T_N)=\dim\overline{(I_{\ell^2(\N)}-T_N(T_N)^*)(\ell^2(\N))}=+\infty$ and thus $\gamma(T_N)=N<\gamma_p(T_N)=+\infty$.

Next, note that these arguments also show that if $T_{\infty}=\lambda I:\ell^2\to\ell^2$ for some $\abs{\lambda}<1$, then $\gamma(T_{\infty})=\gamma_p(T_{\infty})=+\infty$.

The cases where $\gamma_p(T)$ is finite and positive, so that either $0<\gamma(T)=\gamma_p(T)<+\infty$ or $0<\gamma(T)<\gamma_p(T)<+\infty$, can be addressed by using techniques and examples from \cite{CMS}.

Finally, let us remark that it is also possible for an operator $T$ to have a frame of iterations but not a Parseval frame of iterations, i.e. $\gamma(T)>0=\gamma_p(T)$. This situation occurs whenever $T$ is similar to a strongly stable contraction (so $\gamma(T)>0$ by Theorem \ref{coro}) but is not a contraction itself (so $\gamma_p(T)=0$ by Theorem \ref{thm:iterationsParseval}). A possible example is to take $T:\ell^2(\N)\to\ell^2(\N)$ with $Te_1=2e_2$, $Te_2=0$ and $Te_j=\lambda_je_j$ when $j\geqslant3$ and where $\{\lambda_j\}_{j=3}^\infty$ is some Carleson sequence in $\D$. Thus $\norm{T}=2$ and $T$ is not a contraction, but reasoning as before, we can take $f\in\ell^2(\N)\ominus\text{span}\{e_1,e_2\}$ such that $\{T^ne_1\}_{n\geqslant0}\cup\{T^nf\}_{n\geqslant0}$ forms a frame and thus $\gamma(T)>0$.
\end{example}

Frames of iterations for {\em normal} operators with {\em finite index} were completely characterized and all possible collections of generators were obtained in \cite{ACCMP17,ACMT,AP17,CMPP20,CMS21}. From those results and reasoning like in Example \ref{example} we can conclude that a normal operator has Parseval index either $0$ or infinite, or in other words that a normal operator can never have a finite number of vectors generating a Parseval frame of iterations.

The fact that in the previous example we have $\norm{T}=1$ is a consequence of $T$ having a Parseval frame of iterations on the one hand, and, on the other hand, a general frame of iterations with finite generators. Indeed, in \cite{AP17} it was shown that frames of iterations generated by a finite set of vectors force the operator to have norm at least $1$.
Altogether, we have the following:

\begin{proposition}\label{norma1}
Let $\HH$ be an infinite dimensional Hilbert space and $T\in\cF^p_\HH$. Assume further that $\gamma(T)<+\infty$. Then $\norm{T}=1$.
\end{proposition}
\begin{proof}
On one hand, if $\gamma(T)<+\infty$, \cite[Theorem 9]{AP17} says that $\norm{T}\geqslant1$. On the other hand, if $T\in\cF^p_\HH$, Proposition \ref{prop:parseval-ida} says that $\norm{T}\leqslant1$. Altogether, $\norm{T}=1$.
\end{proof}

At this point it is worth noting that there exist operators with norm $1$ but which do not admit {\em any} frame of iterations. Indeed, if $H^2:=H^2(\mathbb T)$ is the Hardy space without multiplicity, take $T=S^*:H^2\to H^2$ the adjoint of the shift in $H^2$. Then $\norm{T}=1$ but $T^*=S$ is not strongly stable and hence, by Theorem \ref{existencia}, $T$ does not admit a frame of iterations.

Moreover, even if $T$ admits a frame of iterations and $\norm{T}=1$ it might happen that $T$ has infinite index. This can be seen, for example, by taking $T$ as the shift $S$ in a Hardy space $H^2_\KK$, where $\KK$ is a Hilbert space with infinite dimension.

\subsection{Linear independence of generators}
Taking up from the result of Theorem \ref{nikc2} that says the Parseval index can always be attained on a linearly independent set of vectors, we now move to study whether this also happens for the general index. Additionally, we address the question about the necessity of linear independence for a generator set which has minimal size. 

To that end, we begin with a result that shows we can always shrink the size of a set of generators to a set of linearly independent vectors that are also generators.
\begin{proposition}\label{lig}
Let $T\in\cB(\HH)$ and $\{v_i\}_{i\in I}\subseteq\HH$ such that $\{T^nv_i\}_{i\in I,n\geqslant0}$ is a frame (Parseval frame) of $\HH$. Then there exists a linearly independent set of vectors $\{w_j\}_{j\in J}$ with $\#J\leqslant\#I$ and such that $\{T^nw_j\}_{j\in J,n\geqslant0}$ is a frame (Parseval frame). 

Moreover, the vectors $\{w_j\}_{j\in J}$ can be chosen so that $\osp\{w_j:j\in J\}=\osp\{v_i:i\in I\}$.
\end{proposition}
\begin{proof}
By item $(i)$ of Proposition \ref{basic} we know that there exists a model space $N\subseteq H^2_{\ell^2(I)}$ and an isomorphism $V:N\to\HH$ such that $v_i=V(P_Ne_i)$ for all $i\in I$, where $\{e_i\}_{i\in I}$ is the canonical basis of $\ell^2(I)$. 

Now, set $\KK:=\ell^2(I)$, $\KK_0:=\KK\cap N^\perp$ and $\KK_1:=\KK\ominus\KK_0$. Next, construct another orthonormal basis of $\KK$ as $\{f_i\}_{i\in I_0}\cup\{g_j\}_{j\in J}$ with $\{f_i\}_{i\in I_0}$ an orthonormal basis of $\KK_0$ and $\{g_j\}_{j\in J}$  an orthonormal basis of $\KK_1$. Then\begin{equation*}
\#J=\dim\KK_1\leqslant\dim\KK=\#I
\end{equation*}
and since $P_N(f_i) = 0$ for $i\in I_0$ we have that
 $\{(A_N)^nP_Ng_j\}_{j\in J}$ is a Parseval frame in $N$ with at most $\#I$ generators. Therefore, defining for each $j\in J$, $w_j:=V(P_Ng_j)$, we get that
\begin{equation*}
\{T^nw_j\}_{j\in J,n\geqslant0}=\{V((A_N)^nP_Ng_j)\}_{j\in J,n\geqslant0}
\end{equation*}
is a frame of $\HH$ with at most $\#I$ generators. Also, since $\{g_j\}_{j\in J}$ is a linearly independent subset of $\KK_1$ and $\ker( P_N|_\KK)=\KK_0$, then $P_N$ restricted to $\KK_1$ is one to one, which implies that $\{P_N g_j\}_{j\in J}$ is linearly independent. We conclude that $\{w_j\}_{j\in J}$ is a linearly independent set of vectors.

Moreover, 
\begin{equation*}
\osp\{v_i:i\in I\}=V(\overline{P_N(\KK)})=V(\overline{P_N(\KK_1)})=\osp\{w_j:j\in J\}.
\end{equation*}

Finally notice that if $\{T^nv_i\}_{i\in I,n\geqslant0}$ is Parseval, we can take $V$ to be unitary and thus $\{T^nw_j\}_{j\in J,n\geqslant0}$ would also be Parseval.
\end{proof}

As a corollary we obtain the following result about linear independence for generators with minimal cardinal.
\begin{corollary}
Given an operator $T\in\cB(\HH)$ with positive index, there always exists a linearly independent set of generators of cardinal $\gamma(T).$
Moreover, if the index is finite and positive, every set of generators with cardinal $\gamma(T)$ must be linearly independent.

The same holds for the Parseval index $\gamma_p(T).$
\end{corollary}
\begin{proof}
We will only provide the proof for the index $\gamma(T)$, as the proof for $\gamma_p(T)$ is analogous. If $\{v_i\}_{i\in I}\subseteq\HH$ is a set of generators with $\#I=\gamma(T)$ then by Proposition \ref{lig} we can obtain a set of linearly independent generators $\{w_j\}_{j\in J}$ such that $\#J\leqslant\#I$. However, by definition, $\#J\geqslant\gamma(T)$. In sum, $\#J=\gamma(T)$.

Now if $\gamma(T)$ is finite, take an index set $I$ with $\#I=\gamma(T)$ and suppose that $\{v_i\}_{i\in I}\subseteq\HH$ is a set of generators that is linearly dependent. Then, if $\KK_0$ and $\KK_1$ are as in the proof of Proposition \ref{lig}, we would have $\KK_0\neq\{0\}$ as there would be a non-trivial linear combination of $\{P_Ne_i\}_{i\in I}$ equaling $0$. This means that $\dim\KK_1<\dim\KK=\#I$ and we would then obtain a set of generators with strictly less elements than $\#I$, which is a contradiction.
\end{proof}

In light of this result, we can then make the following definition of {\em optimal frame of iterations (i.e. optimal dynamical frames)}.

\begin{definition}
We say that a frame of iterations $\{T^nv_i\}_{i\in I,n\geqslant0}$ is {\em optimal} when the set of generators $\{v_i\}_{i\in I}$ is linearly independent and $\#I=\gamma(T)$.

We have a corresponding definition for the Parseval frame case.
\end{definition}

Proposition \ref{lig} signals a relation between a frame of iterations and the vector space spanned by its generators. Furthermore, in the Parseval frame case we can show that the dimension of such spaces is always the same. Moreover, it is equal to the Parseval index.

\begin{proposition}\label{span}
Let $T\in\cB(\HH)$ and $\{v_i\}_{i\in I}\subseteq\HH$ such that $\{T^nv_i\}_{i\in I,n\geqslant0}$ is a Parseval frame. Then
\begin{equation*}
\gamma_p(T)=\dim ({\rm span}\{v_i:i\in I\} ).
\end{equation*}
\end{proposition}

To show this we will first prove it when $\HH=N$ a model space and $T=A_N$, which in the sense of Proposition \ref{basic}, represents all cases.

\begin{lemma}
Let $N\subseteq H^2_\KK$ be a model space, then
\begin{equation*}
\overline{(I_N-A_NS^*)(N)}=\overline{P_N(\KK)}.
\end{equation*}
\end{lemma}
\begin{proof}
From the proof of Theorem \ref{nikc2} we know that if $f=\sum_{n\geqslant0}a_n\chi^n\in N$ then
\begin{equation*}
(I_N-A_NS^*)f=P_N(a_0)=P_N(f(0)),
\end{equation*}
identifying $f(0)=a_0$. Thus, if we define $\KK_N:=\overline{\{g(0):g\in N\}}$, 
\begin{equation*}
\overline{(I_N-A_NS^*)(N)}=\overline{P_N(\KK_N)}
\end{equation*}
since $P_N$ is continuous and $\KK_N$ is a closed subspace of $\KK$ (and hence of $H^2_\KK$) .

The inclusion $\overline{P_N(\KK_N)}\subseteq\overline{P_N(\KK)}$ is immediate from $\KK_N\subseteq\KK$. To see the converse inclusion we will show that $P_N(\KK)\subseteq P_N(\KK_N)$. To that end, let $k\in\KK$ and take $k_1:=P_{\KK_N}k$ where $P_{\KK_N}:\KK\to\KK$ is the projection onto $\KK_N$. This means that $k-k_1\perp\KK_N$, and in particular, for all $g\in N$ we have
\begin{equation*}
\scal{k_1-k,g(0)}_\KK=0.
\end{equation*}
Further, since $k_1-k$ is a constant we have that $\scal{k_1-k,g}_{H^2_\KK}=\scal{k_1-k,g(0)}_\KK$ for all $g\in N$. Whence $k-k_1\perp N$, and thus
\begin{equation*}
P_N(k-k_1)=0,
\end{equation*}
or equivalently, $P_N(k)=P_N(k_1)$. This shows that $P_N(k)\in P_N(\KK_N)$ and the proof is finished.
\end{proof}
\begin{proof}[Proof of Proposition \ref{span}]
This actually follows from the proof of Proposition \ref{lig}. Indeed, if $\{A_N^n(P_Ne_i)\}_{i\in I,n\geqslant0}$ is the basic frame unitarily equivalent to $\{T^nv_i\}_{i\in I,n\geqslant0}$ given by $(ii)$ of Proposition \ref{basic}, we have that
\begin{equation*}
V(\overline{P_N(\KK)})=\osp\{v_i:i\in I\}
\end{equation*}
where $V:N\to\HH$ is the unitary map that realizes the equivalence, i.e. $T=VA_NV^*$ and $v_i=V(P_Ne_i)$ for all $i\in I$.

Now, from Theorem \ref{nikc2}, we know that the relation $$T=VA_NV^*$$ implies
\begin{equation*}
\gamma_p(T)=\dim\overline{(I_N-A_NS^*)(N)}
\end{equation*}
And from the previous lemma we have that $\overline{(I_N-A_NS^*)(N)}=\overline{P_N(\KK)}$. Altogether, 
\begin{equation*}
\gamma_p(T)=\dim\osp\{v_i:i\in I\}
\end{equation*}
and the proof is finished.
\end{proof}

To conclude this section, we show that for any bounded operator 
$T\in\cF_{\HH}$ it is possible to explicitly construct a contraction 
$Q$ on $\HH$ that is similar to $T$ and satisfies $\gamma(T)=\gamma_p(Q)$. 
Although this construction does not directly assist in computing the index 
of $T$, it is nonetheless of independent interest.

Let $T$ be a bounded operator on $\HH$. Assume that $T$ is similar to a contraction and that $T^*$ is strongly stable. Then there exists a minimal, linearly independent collection of vectors 
$G=\{g_j\}_{j\in J}\subset\HH$ such that 
\begin{equation}\label{marco}
\{T^n g_j:j\in J,n\geqslant0\}
\end{equation}
is a frame of $\HH$. Hence, the index of $T$ is given by $\gamma(T)=\#J$. Let $S:\HH\to\HH$ be the frame operator associated with the system in \eqref{marco}, that is, 
\begin{equation*}
Sf=\sum_{j,n}\scal{f, T^n g_j}T^ng_j,\quad\text{for every }f\in\HH.
\end{equation*}
It is well known that if $\{h_j\}_j$ is a frame of $\HH$, then $\{S^{-1/2}h_j\}_j$ is a Parseval frame for $\HH$ (see for instance, 
\cite[Corollary 8.28]{Heil}). Applying this result to our frame, we obtain that 
\begin{equation*}
\{(S^{-1/2} T S^{1/2})^n (S^{-1/2} g_j)\}_{j,n}
\end{equation*}
is a Parseval frame of iterations in $\HH$. We set $Q:=S^{-1/2} T S^{1/2}$.
\begin{proposition}
With the above notation, we have that
\begin{equation*}
\gamma(T)=\gamma_p(Q)=\dim\overline{(I-QQ^*)(\HH)}.
\end{equation*}
\end{proposition}
\begin{proof}
We begin by noting that 
\begin{equation*}
I-QQ^*=S^{-1/2} (S-T S T^*) S^{-1/2}.
\end{equation*}
Hence, $\dim\overline{(I-QQ^*)(\HH)}=\dim\overline{(S-T S T^*)(\HH)}$, where we have used that $S^{-1/2}$ is an isomorphism.

For every $f\in\HH$, we compute:
\begin{align*}
TST^*f&=T(S(T^* f))=T\left(\sum_{j\in J,n\geqslant0}\scal{T^*f,T^n g_j}T^n g_j\right)
\\
&=\sum_{j\in J,n\geqslant1}\scal{f,T^ng_j}T^n g_j.
\end{align*}
Thus,
\begin{equation*}
(S-T S T^*)f=\sum_{j\in J}\scal{f,g_j}g_j=S_G(f),
\end{equation*}
where $S_G$ denotes the frame operator of the family $\{g_j\}_{j\in J}$.

Assume first that $\#J <\infty$. Since the vectors $\{g_j:j\in J\}$ are linearly independent, they form a basis of $K=\text{span}\{g_j:j\in J\}$, and hence a frame for $K$. Therefore, $S_G$ is an isomorphism from $K$ onto itself, and
\begin{equation*}
\dim (S_G(K))=\#J=\gamma(T).
\end{equation*}

Now, if $\#J=\gamma(T)$ is infinite, then $\gamma_p(Q)$ must also be infinite. Indeed, if a finite collection $\{S^{-1/2} T^n h_j\}_{j\in L}$ with $L$ finite were a Parseval frame of $\HH$, then $\{T^n h_j\}_{j\in L}$ would also form a frame of $\HH$, contradicting the assumption that the frame index of $T$ is infinite. Hence, $\gamma_p(Q)=+\infty=\gamma(T)$.
Therefore 
\[
\gamma_p(Q)=\dim\overline{(I-QQ^*)(\HH)} = \gamma(T)
\]

which completes the proof.
\end{proof}

Furthermore, since the operator $Q=S^{-1/2} T S^{1/2}$ generates a Parseval frame of iterations, it must necessarily be a contraction. We now provide a direct proof of this fact.

Notice that if $S$ is the frame operator of the system \eqref{marco} and $S_G$ that of $G$, then
\begin{equation*}
S-T S T^*=S_G.
\end{equation*}
Proving that $Q$ is a contraction is equivalent to showing that $Q^*$ is one; we verify the latter. For any $f\in\HH$,
\begin{align*}
\norm{f}^2-\norm{Q^* f}^2&=\scal{f,f}-\scal{Q^* f, Q^* f}
\\
&=\scal{f,f}-\scal{S^{1/2}T^*S^{-1/2}f,S^{1/2} T^* S^{-1/2} f}
\\
&=\scal{f,f}-\scal{S^{-1/2}TST^*S^{-1/2}f,f}
\\
&=\scal{S^{-1/2}(S-T S T^*)S^{-1/2}f,f}
\\
&=\scal{S_G S^{-1/2}f,S^{-1/2}f}\geqslant0,
\end{align*}
where the inequality follows from the positivity of $S_G$. Hence, $\norm{Q^* f}\leqslant\norm{f}$ for all $f\in\HH$, and therefore $Q$ is a contraction.

Another useful property of the operator $Q$ is that the kernels of the synthesis operators (see Definition \ref{synthesis} for a precise definition) associated with the frames $\{Q^n(S^{-1/2} g_j)\}_{j\in J,n\geqslant0}$ and  $\{T^n g_j\}_{j\in J,n\geqslant0}$ coincide. Consequently, both operators are represented on the same model space $N$, which captures the structure of their action. The distinction between them lies in their relation to the operator $A_N$, the compression of the shift to the model space $N$: the operator $T$ is {\em similar} to $A_N$, whereas $Q$ is {\em unitarily equivalent} to it.

\section{Strongly Stable Contractions with Strongly Stable Adjoints}\label{full range}

In this section, we focus on contractions acting on a general separable Hilbert
space, possessing the {\em joint strong stability} property, namely that both the operator and its adjoint are strongly stable. This property implies that $T$ and $T^*$ admit frames of
iterations.

We first work in $H^2_\KK$. For every model space $N\subseteq H^2_\KK$, the adjoint of the compressed shift $A_N$ coincides with the backward shift $S^*$, which is strongly stable. Then, the {\em joint} strong stability of $A_N$ reduces to requiring only the strong stability of $A_N$. 

We prove that the strong stability of the shift compression $A_N$ is equivalent to the orthogonal complement of $N$ in $H^2_{\KK}$ being a {\em full range subspace}, a concept studied by Helson in the 1960s in the context of shift-invariant spaces.

\subsection{Full range subspaces and inner functions }\label{fullranges}
The characterization of model subspaces in a Hardy space $H^2_\KK$ comes from the description of their orthogonal complement. Such complements are, by definition, invariant under the unilateral shift $S$. 

Before presenting this characterization, we need to introduce a few definitions.
\begin{definition}\label{inner}
Let $Q:\T\rightarrow\cB(\KK)$ be an operator-valued measurable function such that

\begin{equation*}
\norm{Q}_\infty:=\esssup_{z\in\T}\norm{Q(z)}_{op}<\infty.
\end{equation*}
We define the operator $\widehat{Q}:L^2(\T,\KK)\rightarrow L^2(\T,\KK)$ by
\begin{equation*}
(\widehat{Q}f)(z):=Q(z)f(z)\qquad\text{a.e. }z\in\T,f\in L^2(\T,\KK).
\end{equation*}
Note that $\widehat{Q}$ is well-defined and bounded with $\norm{\widehat{Q}}_{op}\leqslant\norm{Q}_\infty$.

We will say that $Q$ is {\em analytic} if $\widehat{Q}(H^2_\KK)\subseteq H^2_\KK$. Denote by $\cA=\cA(\KK)$ the set of all {\em analytic} operator-valued functions. 
Given $Q\in\cA(\KK)$, we say that $Q$ is  inner if $Q(z):\KK\to\KK$ is a unitary operator for almost every $z\in\T$.

We denote by $$\cA_\cI=\cA_\cI(\KK)$$  the subclass of inner functions in $\cA(\KK)$.
\end{definition}
Being analytic can also be shown through a Fourier-type expansion. To see this, fix $\{e_i\}_{i\in I}$ an orthonormal basis of $\KK$ and let $Q\in\cA(\KK)$. For each $n\geqslant0$, define $Q_n:\KK\to\KK$ to be the operator given by
\begin{equation*}
Q_nv=\sum_{i\in I}\scal{\widehat{Q}v,e_i\rchi^n}e_i,\qquad v\in\KK.
\end{equation*}
Note that $Q_n$ is indeed well-defined and bounded since
\begin{equation}\label{cuentad}
\sum_{i\in I}\abs{\scal{\widehat{Q}v,e_i\rchi^n}}^2\leqslant\sum_{m\geqslant0}\sum_{i\in I}\abs{\scal{\widehat{Q}v,e_i\rchi^m}}^2=\norm{\widehat{Q}v}^2\leqslant\norm{\widehat{Q}}_{op}^2\norm{v}^2_\KK.
\end{equation}
Moreover, as $Q_nv$ is the $n^{th}$-Fourier coefficient of $\widehat{Q}v$, we may write, for almost all $z\in\T$,
\begin{equation}\label{eq:fourier_Q}
Q(z)=\sum_{n\geqslant0}Q_nz^n
\end{equation}
in the strong operator topology, i.e. $Q(z)v=\sum_{n\geqslant0}Q_nvz^n$ for all $v\in\KK$. Indeed, convergence is assured by the right-hand side of \eqref{cuentad}. Further, by uniqueness of Fourier coefficients in $H^2_\KK$ (constants or otherwise), the operators $Q_n$ in this expansion are also unique. We will thus say that $Q_n$ is the $n^{th}$-Fourier (operator) coefficient of $Q$.

Analytic functions are  a key to describe shift-invariant subspaces as the following result shows.

\begin{theorem}[{Beurling-Lax-Halmos Theorem \cite[$\S$VI]{H}}]\label{thm:beu-lax-hal}
Let $\KK$ be a Hilbert space and $M\subseteq H^2_\KK$ a closed subspace. Then $M$ is $S$-invariant if and only if there exists a subspace $\KK_1\subseteq\KK$ and a function $Q\in\cA(\KK)$ such that
\begin{equation}\label{asociada}
M=\widehat{Q}(H^2_{\KK_1})
\end{equation}
and $Q(z):\KK\to\KK$ is a partial isometry with initial space $\KK_1$ for almost every $z\in\T$.
Moreover, if $\KK_2\subseteq\KK$ and $\widetilde{Q}\in\cA(\KK)$ are another pair satisfying \eqref{asociada} then there exists a unitary map $R:\KK_1\to\KK_2$ such that $Q(z)=\widetilde{Q}(z)R$ for almost all $z\in\T$.
\end{theorem}

We remark that in \cite{H}, $\KK_1$ is taken as an {\em arbitrary} separable Hilbert space. But it follows from the fact that there is an isometric embedding of $\KK_1$ into $\KK$ that one can take $\KK_1$ as a subspace of $\KK$, as we do here.

In what follows we will work with a subclass of shift-invariant subspaces called {\em full range} that we define now.

Let $M$ be an $S$-invariant subspace of $H^2_{\KK}$ and let $W=M\ominus S(M)$ denote its wandering subspace. Consider an orthonormal basis $\{E_j\}_{j\in I}$ of $W$. For almost every $z\in\T,$ define
\begin{equation}\label{range function}
J(z)=\osp\{E_j(z): j\in I\}.
\end{equation}
We note that this definition does not depend on the orthonormal basis $\{E_j\}_{j\in I}$ chosen. The function $J$ defined in the torus $\T$ and taking values in the closed subspaces of $\KK,$ is referred to as an {\em analytic measurable range function}. This concept was introduced by Henry Helson in \cite{H} who conducted a thorough analysis of range functions to characterize $S$-invariant subspaces.

\begin{definition}\label{def:full-range}
A subspace $M\subseteq H^2_\KK$ is called  full range if $M$ is invariant under $S$ and 
its range function as defined in \eqref{range function} is almost everywhere equal to $\KK$.
\end{definition}

Note that if $M=\widehat{Q}(H^2_{\KK_1})$ then
\begin{equation*}
J(z)=Q(z)(\KK_1)
\end{equation*}
holds for almost every $z\in\T$.

There is a way of describing full range subspaces through the function provided in \eqref{asociada} by the Beurling-Lax-Halmos Theorem. In fact, as we will see, in such case the function is {\em inner} in the sense of Definition \ref{inner}.

 We are now ready to present a characterization of full-range $S$-invariant spaces using inner functions. Furthermore, we provide two additional equivalent conditions for a subspace to be full range (see \cite{Fill,H}).

\begin{theorem}\label{lem:full-range}
Let $M\subseteq H^2_\KK$ be an $S$-invariant subspace and $W=M\ominus S(M)$. The following statements are equivalent:
\begin{enumerate}[label=(\roman*),ref=(\roman*)]
\item $M$ is full range.

\item $M=\widehat{Q}(H^2_\KK)$ for some $Q \in \cA_\cI(\KK)$.

\item $L^2(\T,\KK)=\displaystyle\bigoplus_{n\in\Z}U^nW$, where $U$ is the bilateral shift.

\item Every orthonormal basis $\{E_i\}_{i\in I}$ of \,$W$ satisfies that $\{E_i(z)\}_{i\in I}$ is an orthonormal basis of $\KK$ for almost all $z\in\T$.
\end{enumerate}
\end{theorem}

\begin{proof}
$(i)\Rightarrow(ii)$. 
Let $\KK_1\subseteq\KK$ and $Q_1\in\cA(\KK)$ be a subspace and a function associated with $M$ via \eqref{asociada}. Since $M$ is full range we have that $Q_1(z)(\KK_1)=\KK$ for almost every $z\in\T$, which implies that $\KK_1$ and $\KK$ are isometrically isomorphic. Take $R:\KK\to\KK_1$ a fixed isometric isomorphism and set $Q:=Q_1R$, i.e. $Q(z)=Q_1(z)R$ for almost all $z\in\T$. Then, $Q(z):\KK\to\KK$ is an isometry and
\begin{equation*}
Q(z)(\KK)=Q_1(z)R(\KK)=Q_1(z)(\KK_1)=\KK
\end{equation*}
for almost all $z\in\T$. In other words, $Q(z)$ is a unitary operator in $\KK$ for almost all $z\in\T$.

Also, since $R:\KK\to\KK_1$ is an isomorphism and $Q$ is analytic, we have
\begin{equation*}
\widehat{Q}(H^2_\KK)=\widehat{Q_1}(H^2_{\KK_1})\subseteq H^2_\KK,
\end{equation*}
from where it follows that $Q$ is also analytic and $M=\widehat{Q}(H^2_\KK)$.

$(ii)\Rightarrow (iii)$. Since $\widehat{Q}: H^2_\KK\to H^2_\KK$ is a unitary operator that commutes with $S$ and we are assuming $M=\widehat{Q}(H^2_\KK)$, we have that 
\begin{equation*}
M\ominus S(M)=\widehat{Q}(H^2_\KK)\ominus\widehat{Q}(S(H^2_\KK))=\widehat{Q}(H^2_\KK\ominus S(H^2_\KK)).
\end{equation*}
Then, $W=\widehat{Q}(\KK)$. Further as $\widehat{Q}$ defines also a unitary operator on $L^2(\T,\KK)$ that commutes with the bilateral shift $U$ (see \cite[Corollary 3.19]{RR}), we have 
\begin{equation*}
\bigoplus_{n\in\Z}U^nW=\bigoplus_{n\in\Z}U^n\widehat{Q}(\KK)=\widehat{Q}\left(\bigoplus_{n\in\Z}U^n(\KK)\right)=L^2(\T,\KK).
\end{equation*}

$(iii)\Rightarrow(iv)$. Let $\{E_i\}_{i\in I}$ be an orthonormal basis of $W$. By $(iii)$, $\{U^nE_i\}_{ i\in I,n\in\Z}$ is an orthonormal basis of $L^2(\T,\KK)$. To see that $\{E_i(z)\}_{i\in I}$ is an orthonormal set of $\KK$ first note that
\begin{equation*}
\int_\T\scal{E_j(z), E_i(z)}z^{-n}\d z=\scal{E_j,U^n E_i}=\begin{cases}1&j=i\text{ and }n=0\\0&\text{all other cases}\end{cases}.
\end{equation*}
This means that, when $i\neq j$, all the Fourier coefficients of the function $z\mapsto\scal{ E_i(z),E_j(z)}$ are zero and when $i=j$, all the Fourier coefficients are zero except the coefficient $n=0$ which is equal to one. Thus $\scal{E_j(z),E_i(z)}=0$ for almost every $z\in\T$ for $i\neq j$ and $\norm{E_i(z)}=1$ for almost every $z\in\T$.

To prove completeness, suppose that $x\in\KK$ satisfies $\langle x,E_i(z)\rangle=0 $ for almost every $z\in\T$ and for all $i\in I$. Then
\begin{equation*}
0=\int_\T\langle x,E_i(z)\rangle z^{-n}\,d z =\langle x,U^nE_i\rangle
\end{equation*}
for all $n\in\Z$ and $i\in I$. Since $\{U^nE_i\}_{i\in I,n\in\Z}$ is complete in $L^2(\T,\KK)$, we get that the constant function $x$ must be $0$.

$(iv)\Rightarrow(i)$. It is immediate.

\end{proof}

In the next theorem we show that a sufficient condition to construct explicit frames of iteration of the backward shift within a model space $N$ is precisely that $H^{2}_{\KK}\ominus N$ is full range. 

\begin{theorem}\label{frameN}
Let $N\subseteq H^{2}_{\KK}$ be a model space and suppose that $M=H^{2}_{\KK}\ominus N$ is full range. Then, for every orthonormal basis $\{E_i\}_{i\in I}$ of $W$, the system $\{(S^*)^nS^*E_i\}_{i\in I, n\geqslant0}$ is a Parseval frame of $N$.
\end{theorem}
\begin{proof}
First, notice that $S^*E_i\in N$ for all $i\in I$. Indeed, since
\begin{equation*}
\scal{S^* E_i, S^nE_{i'}}=\scal{E_i, S^{n+1} E_{i'}}=0
\end{equation*}
for every $n\geqslant0$ and $i'\in I$, $S^* E_i$ is orthogonal to the orthonormal basis $\{S^n E_{i'}\}_{i'\in I}$ of $M$ and hence $S^*E_i\in N$. Second, as $N$ is an $S^*$-invariant subspace, then the system $\{(S^*)^nS^*E_i\}_{i\in I,n\geqslant0}$ is included in $N$. 

To show that it forms a frame, we use the fact that $(S^*)^n=P_{H^2_\KK}(U^*)^n|_{H^2_\KK}$ for all $n\geqslant 1$, where $U$ is the bilateral shift operator in $L^2(\T,\KK)$. Indeed, for $n=1$, the equality follows from the definitions of $U^*$ and $S^*$. Assume now that it is true for $n>1$. Then 
\begin{align*}
(S^*)^{n+1}&=S^*P_{H^2_\KK}(U^*)^n|_{H^2_\KK}=P_{H^2_\KK}U^*|_{H^2_\KK}P_{H^2_\KK}(U^*)^n|_{H^2_\KK}=(P_{H^2_\KK}UP_{H^2_\KK})^*(U^*)^n|_{H^2_\KK}\\
&=(UP_{H^2_\KK})^*(U^*)^n|_{H^2_\KK}
=P_{H^2_\KK}^*U^*|_{H^2_\KK}(U^*)^n|_{H^2_\KK}=P_{H^2_\KK}(U^*)^{n+1}|_{H^2_\KK}
\end{align*}
since $P_{H^2_\KK}$ is a projection and $H^2_\KK$ is $U$-invariant. 

Using this, we get that for all $i\in I$ and $n\geqslant1$
\begin{equation*}
(S^*)^nE_i=P_N(S^*)^nE_i=
P_NP_{H^2_\KK}(U^*)^n|_{H^2_\KK}E_i=P_N(U^*)^nE_i=P_NU^{-n}E_i.
\end{equation*}
On the other hand, when $n\geqslant0$ we have $U^nE_i=S^nE_i\in M\perp N$ for all $i\in I$ and thus $P_NU^nE_i=0$ for every $i\in I$ and $n\geqslant0$. Therefore, $\{(S^*)^nS^*E_i\}_{i\in I,n\geqslant0}$ is the projection onto $N$ of $\{U^nE_i\}_{i\in I,n\in\Z}$. Since $N\subseteq\bigoplus_{n\in\Z}U^nW$ and $\{U^nE_i\}_{i\in I,n\in\Z}$ is an orthonormal basis of $\bigoplus_{n\in\Z}U^nW$, we conclude that $\{(S^*)^nS^*E_i\}_{i\in I,n\geqslant0}$ is a Parseval frame for $N$.
\end{proof}

\subsection{Strong stability of the shift compression and full range subspaces}

In this section we prove that given a model space $N$ in $H^{2}_{\KK}$ and $A_{N}$ the compression of the shift operator $S$ to $N$, the condition of strong stability of $A_{N}$ is equivalent to require that the subspace $M=H^2_\KK\ominus N$ is full range.

One of the implications can be derived quickly from Theorem \ref{frameN} and Proposition \ref{prop:parseval-ida} as we see immediately.

\begin{proposition}\label{Nfullrange->ANstrongstable} Let $N$ be a model space in $H^{2}_\KK$. If $M=H^2_{\KK}\ominus N$ is full range, then $A_N$ is strongly stable.
\end{proposition}
\begin{proof}
Since $M=H^2_{\KK}\ominus N$ is full range, we know, from Theorem \ref{frameN}, that the set of vectors $\{(S^*)^nS^*E_i\}_{i\in I, n\geqslant0}$ is a Parseval frame for any orthonormal basis $\{E_i\}_{i\in I}$ of $W=M\ominus S(M)$. Thus, $S^*|_N$ admits a Parseval frame of iterations and therefore belongs to $\cF^p_N$. By Proposition \ref{prop:parseval-ida} this implies that $(S^*)^*=A_N$ is strongly stable.
\end{proof}

In order to prove the converse of Proposition \ref{Nfullrange->ANstrongstable}, we make use of the concept of {\em dilations}.
Let $T_1\in\mathcal B(\HH_1)$ and $T_2\in\mathcal B(\HH_2)$. Then, $T_2$ is a {\em dilation} of $T_1$ if $\HH_1\subseteq\HH_2$ and 
$T_{1}^n=P_{\HH_1}T_{2}|_{\HH_1}^n$ for all $n\in \mathbb N$. 
When $T_{2}$ is unitary, we call it {\em unitary dilation}. It is proven in \cite[I, $\S4$, Theorem 4.2]{NFBK10} that for every contraction $T_{1}$, there exists a unitary dilation $T_{2}$ on some Hilbert space $\HH_2$ that is minimal in the sense that $\HH_2$ can be taken to be $\bigoplus_{n\geq0} T_{2}^n\HH_1$.

Now, if $N\subseteq H^2_\KK$ is a model space and $P_N:L^2(\T,\KK)\to L^2(\T,\KK)$ is the orthogonal projection onto $N$, 
it is not difficult to see that $U:L^2(\T,\KK)\to L^2(\T,\KK)$ is a unitary dilation of $A_N$. Moreover, a minimal unitary dilation can be obtained by restricting $U$ to the reducing subspace
\begin{equation*}
\mathcal{G}:=\osp\{U^nN:n\in\Z\}.
\end{equation*}

We need the following result.

\begin{theorem}[{\cite[II, Theorem 1.2]{NFBK10}}]\label{nagy} Let $N$ be a model space in $H^{2}_\KK$ and
let $\mathcal{L}:=\overline{(U-A_N)(N)}$. Then $A_N$ is strongly stable if and only if $\mathcal{G}=\bigoplus_{n\in\Z}U^n\mathcal{L}$.
\end{theorem}

We now have the following theorem.

\begin{theorem}\label{Woplus}
Let $N$ be a model space in $H^{2}_\KK$. If $A_N$ is strongly stable, then $M=H^{2}_{\KK}\ominus N$ is full range.
\end{theorem}
\begin{proof}

By Theorem \ref{lem:full-range}, it suffices to prove that 
$\bigoplus_{n\in\Z}U^nW = L^2(\T,\KK),$ where $W=M\ominus S(M)$ is the wandering subspace of $M$. To this end, we show that $H^2_\KK\subseteq\bigoplus_{n\in\Z}U^nW$, since in that case
\begin{equation*}
L^2(\T,\KK)=\osp\{U^nH^2_\KK:n\in\Z\}\subseteq\bigoplus_{n\in\Z}U^nW\subseteq L^2(\T,\KK).
\end{equation*}

As $H^2_{\KK} = N\oplus M$ and $M = \bigoplus_{n\geqslant0}S^nW \subseteq \bigoplus_{n\in\Z}U^nW$,
it remains only to show that $N\subseteq\bigoplus_{n\in\Z}U^nW$.

Let $\mathcal{L}:=\overline{(U-A_N)N}$ be as in Theorem~\ref{nagy}.
Since $A_N$ is strongly stable, Theorem~\ref{nagy} yields
\begin{equation*}
N\subseteq\osp\{U^nN:n\in\Z\}=\bigoplus_{n\in\Z}U^n\mathcal{L}.
\end{equation*}
Hence it suffices to show that $\mathcal{L}\subseteq M$.
For every $f\in N$ we have
\begin{equation*}
(U-A_N)f = Uf - P_NUf = (I_{H^2_\KK}-P_N)Uf = P_MUf\in M,
\end{equation*}
so $(U-A_N)N\subseteq M$, and therefore $\mathcal{L}\subseteq M$.
Consequently, $\bigoplus_{n\in\Z}U^n\mathcal{L}\subseteq\bigoplus_{n\in\Z}U^nM=\bigoplus_{n\in\Z}U^nW$,
which completes the proof.
\end{proof}
We summarize our results in the following theorem.

\begin{theorem}\label{existence}
Let $N$ be a model space in $H^{2}_\KK$. The following statements are equivalent:
\begin{enumerate}[label=(\roman*),ref=(\roman*)]
\item The subspace $M=H^{2}_{\KK}\ominus N$ is full range.
\item The shift compression $A_{N}$ is strongly stable.
\item For every orthonormal basis $\{E_j\}_{j\in I}$ of the wandering subspace of $H^2_{\KK}\ominus N$, the system $\{(S^*)^n (S^*E_j)\}_{j\in I,n\geqslant0 }$ is a Parseval frame of $N$.
\end{enumerate}

\end{theorem}

Finally, we can characterize the joint strong stability property for general contractions as follows:

\begin{theorem}
Let $T\in\mathcal B(\HH)$ be a contraction. Then, $T$ has the joint strong stability property if and only if there exists a model space $N\subseteq H^2_\KK$ satisfying that $H^2_{\KK}\ominus N$ is full range, such that $T$ is similar to $A_N$.

\end{theorem}
\begin{proof}
First, assume that $T$ has the joint strong stability property. Then, by Theorem \ref{niko}, there exists a model space $N\subseteq H^2_\KK$ such that $T$ is similar to $A_N$. Therefore, $A_N$ must have the joint strong stability property and thus, by Theorem \ref{existence}, $H^2_{\KK}\ominus N$ is full range.

Conversely, suppose that $T$ is similar to $A_N$ with $H^2_{\KK}\ominus N$ being full range. Then, again by Theorem \ref{existence}, $A_N$ has the joint strong stability property. So, the same holds for $T$. 

\end{proof}

\section{Optimal Parseval frames of iterations}\label{optimal}

In this section, we address the construction of optimal Parseval frames of iterations for contractions, namely, Parseval frames of the form $\{T^n v_i\}_{i\in I,n\geqslant0}$ where $I$ has cardinality equal to the Parseval index of $T$.

Moreover, for a contraction $T$ with joint strong stability, we show that the Parseval index of $T$ coincides with that of its adjoint $T^*$, and we provide a simultaneous construction of optimal Parseval frames of iterations for both $T$ and $T^*$.

We begin by working in the model space setting, where we construct Parseval frames of iterations for the shift compression operator and its adjoint using the minimal number of generators. This construction is then transferred to obtain optimal Parseval frames of iterations for $T$ and $T^*$.

\subsection{Optimal Parseval frames in model spaces}

To clarify our exposition, we now fix the setting in which all subsequent results will take place. Given $N\subseteq H^2_{\KK}$ a model space, we define:

\begin{itemize}
\item The $S$-invariant orthogonal complement of $N$, namely $M:= H^2_{\KK}\ominus N$, which we assume to be full range;

\item The wandering subspace of $M$, given by $W:= M\ominus SM$;

\item The subspace $W_1:=W\ominus (W\cap\KK)$;

\item The subspace $\KK_1:=\KK\ominus (W\cap\KK)$;
\item The subspace $\mathcal{L}:=\overline{(U-A_N)N}$.
\end{itemize}

These assumptions and the notation introduced above will serve as the standing hypotheses for the remainder of this section.

\begin{lemma}\label{reduction}
$\dim(W_1)=\dim(\KK_1)$.
\end{lemma}
\begin{proof}
Let $\KK_0= W\cap\KK$. We first note that $H^2_{\KK}=H^2_{\KK_0}\oplus H^2_{\KK_1}.$

On the other hand, since $W$ is the wandering subspace of $M$ and $W=\KK_0\oplus W_1$ we have that 
\begin{equation}\label{wandering-decoposition}
M=\bigoplus_{n\geqslant 0} S^nW=\bigoplus_{n\geqslant 0} S^n\KK_0\oplus\bigoplus_{n\geqslant 0} S^n W_1= H^2_{\KK_0}\oplus M_1,
\end{equation}
with 
$M_1=\bigoplus_{n\geqslant 0} S^nW_1$.

As $H^2_{\KK}=M\oplus N$, using \eqref{wandering-decoposition}, we obtain the decomposition $H^2_{\KK}=H^2_{\KK_0}\oplus M_1\oplus N$. Combining this fact with the first observation, we conclude that 
\begin{equation*}
H^2_{\KK_1}=M_1\oplus N.
\end{equation*}

Now, by Theorem \ref{existence} we know that $A_N$ is strongly stable, which is a property that only depends on $N$. Therefore, using again Theorem \ref{existence}, we get that $M_1$ is a full range $S$-invariant subspace of $H^2_{\KK_1}$. Then, from item $(iv)$ of Theorem \ref{lem:full-range} applied to $M_1$, we obtain $\dim(W_1)=\dim(\KK_1)$ as we wanted. 
\end{proof}

\begin{remark}
We note that, by the proof of Lemma \ref{reduction}, one can always reduce the
Hardy space $H^2_{\mathcal K}$ to $H^2_{\mathcal K_1}$, while the
model space $N \subseteq H^2_{\mathcal K}$ remains a model space in
$H^2_{\mathcal K_1}$. In particular, every element of $N$ admits a power series
expansion whose coefficients belong to $\KK_1$. Moreover, $\KK_1$ is minimal
with this property, in the sense that if $N \subseteq H^2_E$ for some closed
subspace $E \subseteq \KK$, then necessarily $\KK_1 \subseteq E$.

We also observe, using Theorem \ref{existence}, that $M$ is full range if and only if $A_N$ is strongly stable if and only if $M_1$ is full range.

Finally, since $W=(W\cap \KK)\oplus W_1=\KK_0\oplus W_1$, Lemma
\ref{reduction} gives
\[
\dim(W)=\dim(\KK_0)+\dim(W_1)
       =\dim(\KK_0)+\dim(\KK_1)
       =\dim(\KK),
\]
which  can also be deduced from Theorem \ref{lem:full-range}.
\end{remark}

In the next proposition, we construct Parseval frames of iterations in $N$ for the operators $A_N$ and $S^*$.
We shall later prove that these frames are optimal.
\begin{proposition}\label{W1}
 Let $\{E_j\}_{j\in J}$ be an orthonormal basis of $W_1$ and $\{e_j\}_{j\in J}$ an orthonormal basis of $\KK_1$. 
Then, $\{A_N^n (P_Ne_j)\}_{j\in J,n\geqslant 0}$ and
$\{(S^*)^n S^*E_j\}_{j\in J,n\geqslant 0}$ are Parseval frames of $N.$
\end{proposition}

\begin{proof}
We first notice that by Lemma \ref{reduction}, it is valid to take the same index set $J$ for the orthonormal bases in $W_1$ and $\KK_1$. 

Now, let $\{F_i\}_{i\in I}$ be an orthonormal basis of $W\ominus W_1$, so $\{E_j\}_{j\in J}\cup\{F_i\}_{i\in I}$ is an orthonormal basis of $W.$
We know by Theorem \ref{frameN} that $\{(S^*)^n S^*E_j\}_{ j\in J,n\geqslant 0}\cup\{(S^*)^n S^*F_i\}_{ i\in I,n\geqslant 0}$ is a Parseval frame of $N.$

Since $F_i \in W \cap \,\mathcal{K}$ implies $S^*F_i = 0$ for all $i \in I$, it follows that $\{(S^*)^n S^*E_j\}_{j\in J, n\geqslant 0}$ is also a Parseval frame for $N$.
 
The proof for the frame of iterations generated by $A_N$ is analogous. We extend the orthonormal basis $\{e_j\}_{j\in J}$ of $\KK_1$ to an orthonormal basis of $\KK$, say $\{e_j\}_{j\in J}\cup\{f_i\}_{i\in I}. $
Then we observe that $P_Nf_i = 0$ for $i\in I $ since $f_i$ are in $W.$ 
Thus, $\{A^n_N(P_Ne_j)\}_{ j\in J,n\geqslant 0}$ is a Parseval frame of $N$. 
\end{proof}

In what follows, we will show that the Parseval frames of iterations by $A_N$ and $S^*$ constructed on Proposition \ref{W1} are optimal. For this it will be enough to show that $\gamma_p(A_N) =\gamma_p(S^*|_N) =\dim(W_1)=\dim(\KK_1)$.  We need first the following proposition.

\begin{proposition}\label{dimcoro}
For the Parseval index of $S^*|_N$ and $A_N$ we have that
\begin{equation*}
\gamma_p(S^*|_N)=\gamma_p(A_N)=\dim\mathcal{L}.
\end{equation*}
\end{proposition}

\begin{proof}
From Theorem \ref{nikc2} applied to the operator $A_N$ and the operator $S^*$ acting in $N$ we have that
\begin{gather*}
\gamma_p(A_N)=\dim\overline{(I-A_NS^*)(N)},
\\
\gamma_p(S^*|_N)=\dim\overline{(I-S^*A_N)(N)}.
\end{gather*}
On the other hand, \cite[II, Theorem 1.1]{NFBK10} says that
\begin{gather*}
\dim\overline{(U-A_N)(N)}=\dim\overline{(I-A_NS^*)(N)},
\\
\dim\overline{(U^*-S^*)(N)}=\dim\overline{(I-S^*A_N)(N)}.
\end{gather*}
Finally, as $(A_N)^*=S^*$ and both $A_N$ and $S^*$ are strongly stable, \cite[II, Theorem 1.2]{NFBK10} states that $\dim\overline{(U-A_N)N}=\dim\overline{(U^*-S^*)N}$. Altogether,
$\gamma_p(S^*|_N)=\gamma_p(A_N)=\dim\mathcal{L}$.
\end{proof}

\begin{proposition}\label{dimcoro2}
Under the standing setting of this section, we have $\mathcal{L} = W_1$.\end{proposition}
\begin{proof}
From Proposition \ref{dimcoro} we know that $\gamma_p(S^*|_N)=\dim\overline{(U-A_N)(N)}$. Note that
\begin{equation*}
\overline{(U-A_N)(N)}=\overline{\text{rg}(U-P_NU)P_N}
\end{equation*}
where the operator $(U-P_NU)P_N$ is understood as acting in $L^2(\T,\KK)$. Noting that $(U-P_NU)P_N=(I-P_N)UP_N=P_{N^\perp} UP_N$ where $N^\perp=L^2(\T,\KK)\ominus N$ and $P_{N^\perp}$ is the projection in $L^2(\T,\KK)$ onto $N^\perp$, we have that
\begin{equation*}
\overline{\text{rg}(U-P_NU)P_N}=\ker(P_NU^*P_{N^\perp})^\perp.
\end{equation*}

Let us now study the kernel of $P_NU^*P_{N^\perp}:L^2(\T,\KK)\to L^2(\T,\KK)$. To that end, we decompose $L^2(\T,\KK)$ as follows
\begin{equation*}
L^2(\T,\KK)=\left(L^2(\T,\KK)\ominus H^2_\KK\right)\oplus N\oplus S(M)\oplus W
\end{equation*}
where, recall, $M=H^2_\KK\ominus N$ and $W=M\ominus S(M)$. 

First, note that since $L^2(\T,\KK)\ominus H^2_\KK$ is a $U^*$-invariant subspace of $N^\perp$ then
\begin{equation*}
U^*P_{N^\perp}(L^2(\T,\KK)\ominus H^2_\KK)=U^*(L^2(\T,\KK)\ominus H^2_\KK)\subseteq L^2(\T,\KK)\ominus H^2_\KK\subseteq N^\perp,
\end{equation*}
and thus $P_NU^*P_{N^\perp}(L^2(\T,\KK)\ominus H^2_\KK))=\{0\}$. 

Second, the inclusion $N\subseteq\ker(P_NU^*P_{N^\perp})$ follows from $P_{N^\perp}(N)=\{0\}$. 

Third, note that as $S(M)=U(M)\subseteq N^\perp$ then
\begin{equation*}
U^*P_{N^\perp}(S(M)) =U^*U(M)=M,
\end{equation*}
and thus $P_NU^*P_{N^\perp}(S(M))=\{0\}$.

Finally, let $f\in W$ and note that
\begin{equation*}
P_NU^*P_{N^\perp} f=P_NU^*f=P_NS^*f=S^*f
\end{equation*} 
because $W\subseteq N^\perp$, $S^*(W)\subseteq N$ and $N$ is $S^*$-invariant. Thus, if we decompose $W=\KK_0\oplus W_1$, where as before $\KK_0=W\cap\KK$, and we recall that $\ker S^*=\KK$, we have, altogether,
\begin{equation*}
\ker(P_NU^*P_{N^\perp})=\left(L^2(\T,\KK)\ominus H^2_\KK\right)\oplus N\oplus S(M)\oplus \KK_0.
\end{equation*}

It follows that $\ker(P_NU^*P_{N^\perp})^\perp=W_1$ and the proof is finished.
\end{proof}

\begin{corollary}\label{coroOptimal}
The Parseval frames constructed in Proposition \ref{W1} are optimal.
\end{corollary}
\begin{proof}
 
We first notice that if $\{e_j\}_{j\in J}$ and $\{E_j\}_{j\in J}$ are orthonormal bases of $\KK_1$ and $W_1$ respectively then $P_Ne_j$ and $S^*E_j$ are non-zero for all $j\in J.$

This is true because for $ j\in J$ we have $e_j\in\KK_1=\KK\ominus (W\cap\KK)$. Then $e_j\notin W$ which implies that $e_j\notin M= H^2_{\KK}\ominus N$.
Thus $P_Ne_j\neq 0.$

Also since $E_j\in W_1$, then $E_j\notin\KK =\ker(S^*)$, so $S^*E_j\neq 0.$

Now, from Lemma \ref{reduction} and Propositions \ref{dimcoro} and \ref{dimcoro2} we know that 
$$\gamma_p(A_N) =\gamma_p(S^*|_N)=\dim(\mathcal{L})=\dim (W_1)=\dim(\KK_1).$$
Altogether, we have,
 $$\gamma_p(A_N)=\dim(\KK_1) =\#\{e_j: j\in J\}\text{ and }\gamma_p(S^*|_N)=\dim(W_1) =\#\{E_j: j\in J\}.$$
 Therefore, the Parseval frames constructed in Proposition \ref{W1} use the minimal possible number of generators and are hence optimal.
\end{proof}
\subsection{Construction of optimal frames for \texorpdfstring{$T$}{T} and \texorpdfstring{$T^*$}{T*}.}\label{sec:optimalframes}

This subsection presents a construction of optimal Parseval frames of iterations for $T$ and $T^*$ where $T$ is a contraction with the joint strong stability property. Specifically, for
a strongly stable contraction $T$, with strongly stable adjoint $T^*$ we derive vector collections $\{v_j\}_{j\in J}\subseteq\HH$ and 
$\{w_j\}_{j\in J}\subseteq\HH$ such that the sequences of iterations $\{T^nv_j\}_{j\in J,n\geqslant0}$ and $\{(T^*)^nw_j\}_{j\in J, n\geqslant0}$ are Parseval frames of
 $\HH$ with $\gamma_p(T)=\#J=\gamma_p(T^*).$

First, notice that according to Theorem \ref{niko} we know that there exists a unitary operator $V: N\rightarrow\HH$ 
where $N$ is a model space in $H^2_\KK$ that satisfies that 
\begin{equation}\label{eq: T-Testrella}
T = VA_NV^{-1}\text{ and } T^*= VS^*|_N V^{-1},
\end{equation}
and $\KK =\overline{(I_\HH-TT^*)^{1/2}(\HH)}$ is the defect space of $T^*$.
 Since $T^*$ is strongly stable, $A_N$ is strongly stable and by Theorem \ref{Woplus} this implies that $M=H^2_{\KK}\ominus N$ is full range.
Thus, by item $(iv)$ of Theorem \ref{lem:full-range}, $\dim(W)=\dim(\KK)$
where $W=M\ominus S(M)$. Moreover, by the same argument that we used at the end of the proof of Theorem \ref{nikc2}, we have that $W\cap\KK=\{0\}$ because $ W\cap\KK\subseteq M\cap\KK$ and $M\cap\KK=\{0\}$.  Then $W_1=W\ominus(W\cap\KK)=W$.

As a consequence, by Propositions \ref{dimcoro} and \ref{dimcoro2} we have that
\begin{equation*}
\gamma_p(S^*|_N)=\gamma_p(A_N)=\dim\LL=\dim W_1=\dim(W)=\dim(\KK).
\end{equation*}

Since $\dim(W)=\dim(\KK)$, we may choose orthonormal bases  $\{e_j\}_{j\in J}$  of $\KK$ and $\{E_j\}_{j\in J}$  of $W$ indexed by the same $J$ such that the collections
\begin{equation}\label{eqOptimalModel}
\{A_N^n(P_Ne_j)\}_{j\in J,n\geqslant 0}\text{ and }\{(S^*)^n S^*E_j\}_{j\in J,n\geqslant 0} 
\end{equation}
are optimal Parseval frames of $N$ due to Corollary \ref{coroOptimal}.

Now we transfer these frames to $\HH$ using the unitary map $V:N\to\HH$. For this, consider $v_j=V(P_N(e_j))$ and $w_j = V(S^*E_j)$ 
for every $j\in J$. Taking into consideration the relationship \eqref{eq: T-Testrella} and the Parseval frames in \eqref{eqOptimalModel} we conclude that
\begin{equation*}
\{T^n(v_j)\}_{j\in J,n\geqslant 0}\text{ and }\{(T^*)^n (w_j)\}_{j\in J,n\geqslant 0},
\end{equation*}
are Parseval frames of $\HH$.  Finally, observe that $\gamma_p(T)=\gamma_p(A_N)=\gamma_p(S^*|_N)=\gamma_p(T^*)$ so that the frames are optimal. 

While this optimal frame construction works well for contractions, when dealing with an operator $T$ similar to a contraction $C,$ we can construct the optimal frames for $C$ and then use the similarity transformation to get frames for $T$ and $T^*.$
However, the frames obtained for $T$ and $T^*$ 
 via the similarity transformation are typically not optimal and may lack the Parseval property.

\section{Similarity between frames of iterations of \texorpdfstring{$T$ and $T^*$}{T and T*}}\label{estructura}

In this section we examine the similarity between the frames of iterations of an operator $T\in\cB(\HH)$ and its adjoint $T^*$ constructed in Section \ref{sec:optimalframes}. For this we exploit Theorem \ref{cms} giving explicitly the similarity between $T$ and a shift compression in a model space. 

If $\{T^nv_i\}_{i\in I,n\geqslant0}$ is a frame of iterations of $\HH$, where
$I$ is an at most countable index set, and $\{v_i\}_{i\in I}$ is a collection of vectors in $\HH,$ this frame determines its synthesis operator defined as follows.

\begin{definition}\label{synthesis}
The  synthesis operator of the frame of iterations $\{T^n v_i\}_{i\in I,n\geqslant0}$ is the operator $C:H^2_{\ell^2(I)}\to\HH$ defined as
\begin{equation*}
Cf=\sum_{i\in I,n\geqslant0}\langle f,S^ne_i\rangle T^nv_i,
\end{equation*}
where $\{e_i\}_{i\in I}$ is the canonical orthonormal basis of $\ell^2(I)$ and $S$ is the unilateral shift operator acting on $H^2_{\ell^2(I)}$. 
\end{definition}

It is easily seen that $C$ intertwines $S$ with $T$, that is, $T C=C S$ and consequently the subspace $\ker(C)\subseteq H^2_{\ell^2(I)}$ is $S$-invariant. Moreover, as $C$ is surjective, we have that its restriction to the subspace $N:=H^{2}_{\ell^2(I)}\ominus\ker(C)$, which is a model space (since it is $S^*$-invariant), is a bijective map. Also, $C$ satisfies $C(P_Ne_i)=v_i$ for all $i\in I$. Thus, $V:=C|_N$ is a similarity between $T$ and the shift compression $A_N=P_NS|_N$ that maps the frame $\{T^nv_i\}_{i\in I,n\geqslant0}$ in $\HH$ to the frame $\{A_N^n(P_Ne_i)\}_{i\in I,n\geqslant0}$ in $N$.

In short, recognizing that the synthesis operator establishes the similarity and that the representing model space is defined by its kernel allows us to determine when the frames of iterations of an operator and its adjoint are similar. Since frames of iterations of $T$ and $T^*$ are similar to frames of iterations of $A_{N}$ (basic frame) and $S^*$ (\textit{adjoint frame}) respectively, in the same model space $N$, we start by studying the similarity between the two latter.

\subsection{Similarity between frames of iterations of \texorpdfstring{$A_{N}$}{AN} and \texorpdfstring{$S^*$}{S*} in model spaces}

In order to study the similarity between the frames of iterations of $A_N$ and $S^*$ in a model space, we introduce the following definition.

\begin{definition}\label{adjointf}
Let $\{e_i\}_{i\in I}$ be an orthonormal basis of $\KK$, $Q\in\cA_\cI(\KK)$ (see Definition \ref{inner}), and $N=H^2_\KK\ominus\widehat{Q}(H^2_\KK)$. The  adjoint frame of the basic frame $\{A_N^n(P_Ne_i)\}_{i\in I,n\geqslant0}$ is 
\begin{equation*}
\{(S^*)^n S^*E_i\}_{i\in I,n\geqslant0},
\end{equation*}
where $E_i=\widehat{Q}e_i$ for every $i\in I$.
\end{definition}

Notice that since $Q\in\cA_\cI(\KK)$, $\widehat{Q}(H^2_\KK)$ is a full range $S$-invariant space and then, the adjoint frame of the above definition, is the frame constructed in Theorem \ref{frameN} for a particular choice of the orthonormal basis $\{E_i\}_{i\in I}$ of $W$, namely $E_i=\widehat{Q}e_i$ for every $i\in I$.

We address the question of when a {\em basic frame } is similar to its {\em adjoint frame } by providing a complete characterization in terms of the inner functions that define them. 

At this point, we would like to briefly summarize the results and concepts we have obtained so far in an informal way before moving on to the next result.

\subsubsection*{Model spaces, basic frames and adjoint frame}
As a consequence of Theorem \ref{cms}, every frame of iterations is similar to a basic frame $\mathfrak{B}_N=\{A_N^n(P_Ne_i)\}_{i\in I,n\geqslant0}$ in a model space $N\subseteq H_{\KK}^2.$ This basic frame is uniquely determined once a basis of $\KK$ is fixed. 
We denote by 
 $\mathfrak{B}_N^*=\{(S^*)^nS^*E_i\}_{i\in I,n\geqslant0}$ the adjoint frame of $\mathfrak{B}_N$.

Now, given a pair $(\mathcal E,Q)$ where $\mathcal E=\{e_i\}_{i\in I}$ is an orthonormal basis of $\KK$ and $Q\in\cA_\cI(\KK)$, it univocally determines the following objects:

\begin{itemize}
\item $M =\widehat{Q}( H_{\KK}^2)$ a full range $S$-invariant subspace of $H_{\KK}^2$. (see Theorem \ref{lem:full-range});

\item $N = H_{\KK}^2\ominus M$ the model space associated to $M$, which is $S^*$- invariant;

\item $\mathfrak{B}_N =\{A_N^n(P_Ne_i)\}_{i\in I,n\geqslant0}$ the basic frame for $N$;

\item $\mathfrak{B}_N^* =\{(S^*)^nS^*E_i\}_{i\in I,n\geqslant0}$ the adjoint frame of $\mathfrak{B}_N.$
\end{itemize}

The question we now seek to answer is whether a basic frame of iterations $\mathfrak{B}_N$ is similar to its adjoint frame $\mathfrak{B}_N^*.$
Below, we provide an answer. 

We start with the following map, that will be crucial for the characterization:

\begin{definition}
Let $\rho:\cA(\KK)\to\cA(\KK)$ be the mapping given by $\rho(Q)(z):=Q(\overline{z})^*$ for almost every $z\in\T$. 
\end{definition}

\begin{remark}\label{rem:coeficiente-involucion}
\noindent
\begin{enumerate}[label=(\roman*),ref=(\roman*)]
\item The map $\rho$ is well-defined. To see this, first recall that the adjoint of a partial isometry is again a partial isometry. Now let $\{e_i\}_{i\in I}$ be an orthonormal basis of $\KK$ and let $Q\in\cA(\KK)$. Then, given $f\in H^2_\KK$, we have for every $i\in I$ and $n>0$,
$$
\scal{\widehat{\rho(Q)}f,e_i\rchi^{-n}}=\int_\T\scal{Q(\overline{z})^*f(z),e_i\overline{z}^n}_\KK\d z
=\int_\T\scal{f(z),\overline{z}^nQ(\overline{z})e_i}_\KK\d z.
$$
Since $Q\in\cA(\KK)$, the function  $g$ defined by $g(z)=z^nQ(z)e_i$  belongs to  $H^2_\KK$ and since $n>0$, $g(0)=0$. Thus, the function $\tilde g$ given  by $\tilde g(z)= g(\overline{z})$ belongs to $ L^2(\T,\KK)\ominus H^2_\KK$. Hence, $\scal{\widehat{\rho(Q)}f,e_i\rchi^{-n}}=0$ for all $i\in I$ and $n>0$.  Therefore, $\widehat{\rho(Q)}f\in( L^2(\T,\KK)\ominus H^2_\KK)^\perp=H^2_\KK$, proving that $\rho(Q)\in \cA(\KK)$.

\item From the definition, it is clear that $\rho$ is an involution, that is, $\rho(\rho(Q))=Q$.

\item 
The map $\rho$ sends inner functions to inner functions. Indeed, if
$Q\in\cA_\cI(\KK)$, then $\rho(Q)\in\cA(\KK)$ by item (i). Moreover,
$
\rho(Q)(z)=Q(\overline z)^*
$
is unitary for almost every $z\in\T$, since $Q$ is inner. Hence $\rho(Q)$ is also inner.

\end{enumerate}
\end{remark}

We are now ready to prove the main result of this section.

\begin{theorem}\label{kerC_frameS*}
Let $\mathcal E=\{e_i\}_{i\in I}$ be an orthonormal basis of $\KK$, let $Q\in\cA_\cI(\KK)$ and consider the model subspace $N=H^2_\KK\ominus\widehat{Q}(H^2_\KK)$. Let $\mathfrak{B}_N^*=\{(S^*)^nS^*E_i\}_{i\in I,n\geqslant0}$ be the adjoint frame of $\mathfrak{B}_N=\{A_N^n(P_Ne_i)\}_{i\in I,n\geqslant0}$. Then, the kernel of the synthesis operator $C:H^{2}_{\KK}\to N$ of $\mathfrak{B}_N^*$ is given by
\begin{equation*}
\ker C =\widehat{\rho(Q)}H^2_\KK.
\end{equation*}
\end{theorem}

\begin{proof}
A function $f\in H^2_\KK$ belongs to $\ker C$ if and only if for all $m\geqslant 0$ and $j\in I$ we have that $\scal{Cf,S^me_j}=0$. Fix $m\geqslant0, j\in I$ and use the definition of $C$ to get
\begin{align*}
\scal{Cf,S^me_j}&=\sum_{i\in I,n\geqslant0}\scal{f,S^ne_i}\scal{(S^*)^n(S^*E_i),S^me_j}.
\end{align*}

Then, by replacing $E_i =\widehat{Q}e_i$ we have
\begin{equation*}
\scal{Cf,S^me_j}=\sum_{i\in I,n\geqslant0}\scal{f,S^ne_i}\scal{(S^*)^{n+m}(S^*\widehat{Q}e_i),e_j} =\sum_{i\in I, n\geqslant0}\scal{f,S^ne_i}\scal{\widehat{Q}e_i, S^{n+m+1} e_j}.
\end{equation*}

Note that
\begin{equation*}
\scal{\widehat{Q}e_i,S^{n+m+1}e_j}=\int_\T\scal{Q(z)e_i,z^{n+m+1}e_j}\d z=\int_\T\scal{e_i, z^{n+m+1}Q(z)^*e_j}\d z.
\end{equation*}
Next, we will use the following fact: if $g(z)=\sum_{n\in\Z}a_nz^n$ is a scalar function in $L^2(\T)$ then $h(z):=g(\overline{z})=\sum_{n\in\Z}a_n\overline{z}^n=\sum_{n\in\Z}a_{-n}z^n$ has the same $0^{th}$-Fourier coefficient than $g$. Applying this to $g(z)=\scal{e_i, z^{n+m+1}Q(z)^*e_j}$ yields
\begin{align*}
\scal{\widehat{Q}e_i,S^{n+m+1}e_j}&=\int_\T\scal{e_i,\overline{z}^{n+m+1}Q(\overline{z})^*e_j}\d z
\\
&=\int_\T\scal{z^{n+m+1}e_i,Q(\overline{z})^*e_j}\d z
\\
&=\scal{S^{n+m+1}e_i,\widehat{\rho(Q)}e_j}
\\
&=\scal{S^ne_i,(S^*)^mS^*\widehat{\rho(Q)}e_j}.
\end{align*}

Altogether, this gives
\begin{equation*}
\scal{Cf,S^me_j}=\sum_{i\in I,n\geqslant0}\scal{f,S^ne_i}\scal{S^ne_i, (S^*)^{m} S^* (\widehat{\rho(Q)} e_j)}.
\end{equation*}

By using Parseval's identity with the orthonormal basis $\{S^n e_i\}_{i\in I,n\geqslant0}$ we get
\begin{equation*}
\scal{Cf,S^me_j} =\scal{f, (S^*)^{m} S^* (\widehat{\rho(Q)} e_j)}.
\end{equation*}

Therefore
\begin{equation*}
\ker C=\{(S^*)^mS^*(F_j):j\in I,m\geqslant0\}^\perp
\end{equation*}
where $F_i=\widehat{\rho(Q)} e_i\in H^2_\KK$ for every $i\in I$. 

Now, since $Q\in\cA_\cI$, $Q(z)$ is a unitary operator for almost every $z\in\T$ and then, so is $Q(\overline{z})^*$ for almost every $z\in\T$. Thus, as for every $i\in I$, $F_i(z)=Q(\overline{z})^*e_i$ for almost every $z\in\T$, we get that $\{F_i(z)\}_{i\in I}$ is an orthonormal basis of $\KK$ for almost every $z\in\T$.
Further, the following identity holds 
\begin{equation*}
Q(\overline{z})^*Q(z)^{-1} E_i(z)=Q(\overline{z})^*e_i=F_i(z),
\end{equation*}
for almost every $z\in\T$. Therefore, the analytic operator-valued function $V:\T\to\cB(\KK)$ given by $V(z):=Q(\overline{z})^*Q(z)^{-1}$ for almost every $z\in\T$ satisfies that $V(z)$ is a unitary operator that sends the orthonormal basis $\{E_i(z)\}_{i\in I}$ to $\{F_i(z)\}_{i\in I}$ for almost every $z\in\T$.
Then, $\widehat{V}:L^2 (\T,\KK)\to L^2 (\T,\KK)$ defines an isometry on $L^2 (\T,\KK)$ that commutes with the bilateral shift $U$ (in $L^2 (\T,\KK)$) and maps $\{E_i\}_{i\in I}$ to $\{F_i\}_{i\in I}$. 

Denoting by $M=\widehat{Q}(H^2_\KK)$, as $\{E_i\}_{i\in I}$ is an orthonormal basis of $M\ominus S(M)$ and $\widehat{V}(M)$ is $S$-invariant, $\{F_i\}_{i\in I}$ must be an orthonormal basis of $\widehat{V}(M)\ominus S\widehat{V}(M)$. Then, on one hand, by Lemma \ref{lem:wandering-frame}, $\{S^nF_i\}_{i\in I,n\geqslant0}$ is an orthonormal basis of $\widehat{V}(M)$. On the other hand, we can use Theorem \ref{frameN} to conclude that system $\{(S^*)^nS^*(F_i)\}_{i\in I,n\geqslant0}$ is a Parseval frame of $\widehat{V}(M)^{\perp}$. 
Finally, 
\begin{align*}
\ker C&=\{(S^*)^nS^*(F_i):i\in I,n\geqslant0\}^\perp=(\widehat{V}(M)^{\perp})^\perp=\widehat{V}(M)
\\
&=\overline{\text{span}}\{S^nF_i\}_{i\in I,n\geqslant0}=\widehat{\rho(Q)} H^2_\KK.
\end{align*}
This completes the proof.
\end{proof}

As a consequence of Theorem \ref{kerC_frameS*} and Theorem \ref{cms} we obtain the following result.

\begin{corollary}\label{basic_tuple_S*}
Let $\{e_i\}_{i\in I}$ be an orthonormal basis of $\KK$, $Q\in\cA_\cI(\KK)$ and consider the model space $N = H^2_\KK\ominus\widehat{Q}H^2_\KK$. Then, the basic frame associated to $\mathfrak{B}_N^*=\{(S^*)^n S^*E_i\}_{i\in I,n\geqslant0}$
 is given by
$$\mathfrak{B}_{N_\rho}=\{A_{N_\rho}^n(P_{N_\rho}e_i)\}_{i\in I,n\geqslant0}$$
where $N_\rho= H^2_\KK\ominus\widehat{\rho(Q)}H^2_\KK$.
\end{corollary}

Note that, by item $(iii)$ of Remark \ref{rem:coeficiente-involucion}, the $S$-invariant subspace $\widehat{\rho(Q)}H^2_\KK$ is full range. 

The involutive nature of $\rho$ implies that if we start with the basic frame $\mathfrak{B}_{N_\rho}$ and repeat the process of taking the adjoint frame and then finding its associated basic frame, we return to $\mathfrak{B}_{N}$. 
To summarize, let us denote by $\boldsymbol\alpha$ the correspondence that assigns to a frame of iterations, its similar basic frame.
Furthermore, given any basic frame $\mathfrak{B}_{\widetilde N}$ in a model space $\widetilde{N}$, whose orthogonal complement is full range, we denote by 
$\boldsymbol\beta(\mathfrak{B}_{\widetilde N})$ its adjoint frame.

Then, taking the adjoint frame of a basic one and then its associated basic frame, must be performed twice to recover the original basic frame we started with. That is, $(\boldsymbol{\alpha\beta})^2(\mathfrak{B}_N)$ is similar to $\mathfrak{B}_N$ as we show in the following diagram:
\begin{equation*}
\begin{CD}
\mathfrak{B}_{N}=\left\{A_{N}^n (P_{N}e_i)\right\} _{i\in I,n\geqslant0}  @>\boldsymbol\beta>>\mathfrak{B}_{N}^*=\left\{(S^{*})^nS^* (\widehat{Q}e_i)\right\}_{i\in I,n\geqslant 0}\\
@A\boldsymbol\alpha AA  @VV\boldsymbol\alpha V\\
\mathfrak{B}_{N_{\rho}}^*=\left\{(S^{*})^nS^* (\rho(\widehat{Q})e_i)\right\}_{i\in I,n\geqslant 0}  @<\boldsymbol\beta<<  \mathfrak{B}_{N_{\rho}}=\left\{A_{N_{\rho}}^n (P_{N_{\rho}}e_i)\right\} _{i\in I,n\geqslant0} 
\end{CD}
\end{equation*}

\begin{proposition} 
Let $\{e_i\}_{i\in I}$ be an orthonormal basis of $\KK$, $Q\in\cA_\cI(\KK)$ and consider the model space $N=H^2_\KK\ominus\widehat{Q}H^2_\KK$. Then the basic frame associated to $\mathfrak{B}_{N}^*=\{ (S^*)^n S^* (\widehat{Q}e_i)\}_{i\in I,n\geqslant0}$, has an adjoint frame which is similar to $\{A_{N}^n(P_{N}e_i)\}_{i\in I,n\geqslant0}$.
That is, 
$$\boldsymbol\alpha\boldsymbol\beta\boldsymbol\alpha (\mathfrak{B}_{N}^*)=\mathfrak{B}_{N}.$$
\end{proposition}

\begin{proof}
By Corollary \ref{basic_tuple_S*}, the basic frame associated to $\mathfrak{B}_{N}^*$ is $\{A_{N_{\rho}}^n(P_{N_{\rho}}e_i)\}_{i\in I,n\geqslant0}$
where $N_{\rho} = H^2_\KK\ominus\widehat{\rho(Q)}H^2_\KK$. The adjoint frame of the latter is $\{(S^*)^nS^*\widehat{\rho(Q)}e_i\}_{i\in I,n\geqslant0}$.

Again, by Corollary \ref{basic_tuple_S*} the basic frame associated to $\{(S^*)^nS^*\widehat{\rho(Q)}e_i\}_{i\in I,n\geqslant0}$ is built up from the model space $H^2_\KK\ominus\widehat{\rho(\rho(Q))}H^2_\KK$. But, since $\rho$ is an involution, we have that 
\begin{equation*}
\widehat{\rho(\rho(Q))}H^2_\KK=\widehat{Q} H^2_\KK=N, 
\end{equation*}
and thus, the basic frame obtained is $\{A_N^n(P_Ne_i)\}_{i\in I,n\geqslant0}$ as we wanted.
\end{proof}

Due to Corollary \ref{basic_tuple_S*}, the question of when a basic frame and its adjoint frame are similar reduces to asking when $\widehat{Q}$ and $\widehat{\rho(Q)}$ coincide.

To answer this, we need to recall that the partial order between $S$-invariant subspaces of $H^2_\KK$ characterized by inner functions induces a relation between the corresponding inner functions as the following proposition says.

\begin{proposition}[{\cite[Theorem 10]{H}}]\label{inclusion_inner}
Let $Q_1,Q_2\in\cA_\cI$. Then $\widehat{Q}_{1}(H^2_\KK)\subseteq\widehat{Q}_2(H^2_\KK)$ if and only if $Q_2^*Q_1\in\cA_\cI$.
\end{proposition}

Also, in the finite-dimensional case, the next result holds.

\begin{proposition}[{\cite[Corollary of Theorem 11]{H}}]\label{det_const}
Let $\KK$ a finite-dimensional space and $Q\in\cA_\cI$. If the determinant $q:\T\to\C$, $q(z)=\det(Q(z))$ is constant, then so is $Q$. 
\end{proposition}

The final ingredient we need is to see how $\rho$ acts on the Fourier expansion \eqref{eq:fourier_Q}. If $Q\in\cA_\cI$ can be written as
\begin{equation*}
Q(z)=\sum_{n\geqslant0}Q_nz^n
\end{equation*}
where $Q_n$ are bounded operators and the convergence is understood in the strong operator topology, then it follows that
\begin{equation*}
\rho(Q)(z)=Q(\overline{z})^*=\sum_{n\geqslant0}Q_n^*z^n.
\end{equation*}
And since the coefficient-operators of a Fourier expansion are unique, this implies that $Q_n^*$ is the $n^{th}$-Fourier coefficient of $\rho(Q)$ for all $n\geqslant0$.

We now give several characterizations of the case when a basic frame and its adjoint are similar.

\begin{theorem}\label{thm:similarity}
Let $\KK$ be a separable Hilbert space and $Q\in\cA_\cI$ with Fourier coefficients $\{Q_n\}_n$. The following statements are equivalent:
\begin{enumerate}[label=(\roman*),ref=(\roman*)]
\item $\widehat{\rho(Q)}H^2_\KK=\widehat{Q}H^2_\KK$;

\item There exists a unitary operator $A:\KK\to\KK$ such that $AQ_n=Q_nA=(AQ_n)^*$ for all $n\geqslant0$.
\end{enumerate}
Moreover, if $\dim(\KK) <\infty$ then (i) and (ii) are also equivalent to
\begin{enumerate}[label=(\roman*),ref=(\roman*)]
\item[(iii)] $\widehat{\rho(Q)}H^2_\KK\subseteq\widehat{Q}H^2_\KK$;

\item[(iv)] $\widehat{Q}H^2_\KK\subseteq\widehat{\rho(Q)}H^2_\KK$.
\end{enumerate}
\end{theorem}

\begin{proof}
We prove that $(i)\Leftrightarrow(ii)$. Suppose first that $\widehat{\rho(Q)}H^2_\KK=\widehat{Q}H^2_\KK$. Then, by the unicity part of Theorem \ref{thm:beu-lax-hal} there exists a unitary operator $\widehat{R}$ such that $R(z)=R:\KK\to\KK$ almost all $z\in\T$ and $\widehat{\rho(Q)}=\widehat{Q}\widehat{R}$. Applying $\rho$ on both sides yields
\begin{equation*}
\widehat{Q}=\widehat{\rho(\rho(Q))}=\widehat{\rho(R)}\widehat{\rho(Q)}=\widehat{R}^*\widehat{Q}\widehat{R}.
\end{equation*}
Then, since the Fourier coefficients of $Q$ are unique, we get the following relation $Q_n^*=RQ_n=Q_nR$ for all $n\geqslant0$.

Now, since $R$ is unitary, we can apply the functional calculus for normal operators to find a unitary operator $A$ such $A^2=R$ and $AQ_n=Q_nA$ for all $n\geqslant0$. Also
\begin{equation*}
(AQ_n)^*=Q_n^*A^*=(Q_nR)A^*=Q_nA^2A^*=Q_nA=AQ_n
\end{equation*}
holds for all $n\geqslant0$.

Reciprocally, assume there exists a unitary operator $A$ such that $AQ_n=Q_nA=(AQ_n)^*$ for all $n\geqslant0$ and define $R:=A^2$. Then
\begin{equation*}
Q_nR=Q_nA^2=(Q_nA)A=(AQ_n)A=(AQ_n)^*A=Q_n^*A^*A=Q_n^*
\end{equation*}
for all $n\geqslant0$. Hence, $\rho(Q)(z)=Q(z)R$ for almost all $z\in\T$ which implies that $\widehat{\rho(Q)}H^2_\KK=\widehat{Q}H^2_\KK$.

Suppose now that $\dim(\KK) = d$. Then, it is clear that $(i)\Rightarrow(iii)$ and $(i)\Rightarrow(iv)$. For the converses we will only show that $(iv)\Rightarrow(i)$ as the proof for $(iii)\Rightarrow(i)$ is analogous. 

The inclusion $\widehat{Q}H^2_\KK\subseteq\widehat{\rho(Q)}H^2_\KK$ implies by Proposition \ref{inclusion_inner} that $\rho(Q)^*Q$ is inner. 
Set $q(z)=\det(Q(z))$ and note that
\begin{equation*}
\det(\rho(Q)(z)^*Q(z))=\det(Q(\overline{z})Q(z))=q(\overline{z})q(z)
\end{equation*}
for almost every $z\in\T$. The fact that $\rho(Q)^*Q$ is analytic, implies that the $L^2(\T)$-function $f(z)=q(\overline{z})q(z)$ must be analytic. But since $f(z)=f(\overline{z})$ for almost every $z\in\T$, $f$ must be constant. Thus, by Proposition \ref{det_const}, $\rho(Q)^*Q$ must also be constant. In other words, there exists a fixed unitary operator $R$ such that $Q(z)=\rho(Q)(z)R$ for almost every $z\in\T$ and thus $\widehat{\rho(Q)}H^2_\KK=\widehat{Q}H^2_\KK$.
\end{proof}

\subsection{More results in the scalar case}

In this section, we show that, in the scalar-valued case $\KK=\C$, Theorem \ref{thm:similarity} admits two additional equivalent conditions characterizing when a basic frame is similar to its adjoint frame.

Notice that the word `inner' in Definition \ref{inner} is chosen so that it extends that of scalar inner functions. Recall that we say that $q:\T\to\C$ is inner if it is in $H^2$ and $\abs{q(z)}=1$ for almost every $z\in\T$. Thus, trivially, the map $\zeta\mapsto q(z)\zeta$ defines a unitary operator in $\C$ for almost every $z\in\T$. Additionally, since $q\in H^2$ and is bounded, the product $z\mapsto q(z)f(z)$ is in $H^2$ for any $f\in H^2$. 
Altogether, this shows that the scalar inner functions are precisely the elements of $\cA_\cI(\C)$.

Applying Theorem \ref{thm:beu-lax-hal} to this case we obtain the following.

\begin{theorem}[Beurling]\label{thm:beurling}
A closed subspace $M\subseteq H^2$ is $S$-invariant if and only if $M=qH^2$ for some inner function $q$. If $q_1$ and $q_2$ are inner functions such that $q_1 H^2 = q_2 H^2$, then $q_1 = c\, q_2$, for some $c\in\T$.
\end{theorem}

We remark that, in this case the associated function in \eqref{asociada} is always inner and thus all $S$-invariant subspaces of $H^2$ are full range.

The basic types of inner functions are complex constants of module one, 
Blaschke products and singular functions. We recall that a Blaschke product is a function
\begin{equation}
B(z) =\alpha z^{M}\prod_{n\geqslant 1}\frac{\overline{a_{n}}}{|a_{n}|}\frac{a_{n} -z}{1-\overline{a_{n}}z},\quad z\in\D
\end{equation}
where $\alpha\in\T$, $M\geqslant0$ and $\{a_{n}\}_n\subseteq\D\setminus\{0\}$ is a Blaschke sequence, i.e., it satisfies the condition $\sum_{n\geqslant1} (1-|a_{n}|) <\infty$.

Moreover, a {\em singular inner function} has the form
\begin{equation} 
S_{\mu}(z) =\alpha\exp\left( -\int_\T\frac{w+z}{w-z}\,d\mu(w)\right),\quad z\in\D
\end{equation}
where $\alpha\in\T$ and $\mu$ is a positive measure on $\T$ which is singular with respect to the normalized Lebesgue measure on $\T$.

In fact, a theorem due to Nevanlinna and F. Riesz (see \cite[Theorem 2.14]{GMR}) shows that every inner function $q$ can be factorized as $q(z)= z^{k} B(z) S_{\mu}(z)$, where $k\geqslant 0$, $B$ is a Blaschke product and $S_{\mu}$ is a singular inner function. For more details on the structure of inner functions we refer the reader to \cite[Section 2.3]{GMR}.

The $S$-invariant subspaces of $H^2$ can be partially ordered according to the di\-vi\-si\-bi\-li\-ty of inner functions. This is, $q_{1} H^2\subseteq q_2H^2$ if and only if $q_2$ divides $q_{1}$ in the sense that there exists a function $q\in H^2$ such that $q_1 = q_2 q$.

The following lemma will be useful for the proof of Proposition \ref{unidim} and it is a consequence of \cite[Proposition 4.10]{GMR}.
\begin{lemma}\label{blaschke_sing}

\

\begin{enumerate}[label=(\roman*),ref=(\roman*)]

\item If $B_{1}$ and $B_2$ are two Blaschke products, then $B_{1}$ divides $B_2$ if and only if the set of zeros of $B_{1}$ is contained in the set of zeros of $B_2$.

\item If $S_{\mu_1}$ and $S_{\mu_2}$ are two singular functions, then $S_{\mu_1}$ divides $S_{\mu_2}$ if and only if $\mu_2 -\mu_1$ is a positive measure.

\item If $q_{1} = c_{1} B_{1} S_{\mu_1}$ and $q_2 = c_2 B_2 S_{\mu_2}$ are two inner functions with their respective factorizations, then $q_1$ divides $q_2$ if and only if $B_1$ divides $B_2$ and $S_{\mu_1}$ divides $S_{\mu_2}$.
\end{enumerate}
\end{lemma}

Theorem \ref{kerC_frameS*} here reads as follows. Observe that the involution $\rho$ here acts as $\rho(q)(z)=q(\overline{z})$ for any inner function $q$.

\begin{proposition}
Let $q\in H^2$ be an inner function and denote $N=H^2\ominus qH^2$. The kernel of the synthesis operator $C$ for $\{(S^*)^nS^*q\}_{n\geqslant 0}$ is given by
\begin{equation*}
\ker C=\rho(q)H^2
\end{equation*}
where $\rho(q)(z)=\overline{q(\overline{z})}$ for almost all $z\in\T$.
\end{proposition}
Now, we give the analogue of Theorem \ref{thm:similarity} that characterizes when the kernel of the synthesis operators of a basic frame and its adjoint frame coincide. In this case we add two more equivalent conditions to those given in Theorem \ref{thm:similarity}. 

\begin{theorem}\label{unidim}
Let $q\in H^2$ be an inner function. The following statements are equivalent
\begin{enumerate}[label=(\roman*),ref=(\roman*)]
\item $\rho(q)H^2=qH^2$

\item $\rho(q)H^2\subseteq qH^2$

\item $qH^2\subseteq\rho(q)H^2$

\item The set of zeros of $q$ in $\D$ is closed under conjugation counting multiplicities, and the singular measure $\mu$ associated to $q$ is also invariant under conjugation.

\item There exists a constant $\alpha\in\T$ such that $\alpha\,\hat{q}(n)$ is real for all $n\geqslant0$. 
\end{enumerate}
\end{theorem}
\begin{proof}
The equivalences $(i)\Leftrightarrow(ii)\Leftrightarrow(iii)$ were already proven in Theorem \ref{thm:similarity}. To prove $(iii)\Leftrightarrow(iv)$ first note that the inclusion $qH^2\subseteq\rho(q)H^2$ is equivalent to $\frac{q}{\rho(q)}\in H^2$. Next, write $q$ as $q(z)=B(z)S_\mu(z)$ where $B$ is a Blaschke product and $S_\mu$ a singular inner function. Thus $\rho(q)(z)=\rho(B)(z)\rho(B_\mu)(z)$. It is clear that $\tilde{B}(z):=\rho(B)(z)$ is a Blaschke product with the set of zeros been the conjugates of those of $B$ (counting multiplicity), and as for $\rho(S_\mu)(z)$ note that
\begin{multline*}
\rho(S_\mu)(z)=\overline{S_\mu(\bar{z})}=\overline{\exp\left(-\int_\T\frac{\zeta+\bar{z}}{\zeta-\bar{z}}\d\mu(\zeta)\right)}=\exp\left(\overline{-\int_\T\frac{\zeta+\bar{z}}{\zeta-\bar{z}}\d\mu(\zeta)}\right)
\\
=\exp\left(-\int_\T\frac{\overline{\zeta+\bar{z}}}{\overline{\zeta-\bar{z}}}\d\mu(\zeta)\right)=\exp\left(-\int_\T\frac{\bar{\zeta}+z}{\bar{\zeta}-z}\d\mu(\zeta)\right)=\exp\left(-\int_\T\frac{\zeta+z}{\zeta-z}\d\tilde{\mu}(\zeta)\right)
\end{multline*}
where $\tilde{\mu}$ is the push-forward measure by conjugation, i.e. $\tilde{\mu}(A)=\mu(\bar{A})$ for all measurable $A\subseteq\T$. In sum, $\rho(q)(z)=\tilde{B}(z)S_{\tilde{\mu}}(z)$ is the decomposition of $\rho(q)$ into a Blaschke product and a singular inner function. By Lemma \ref{blaschke_sing}, $\frac{q}{\rho(q)}\in H^2$ if and only if the zero set of $B$ is invariant under complex conjugation, counting multiplicities, and the measure $\mu$ is invariant under conjugation as $\mu(A)-\mu(\bar{A})\geqslant0$ for all measurable subset $A\subseteq\T$, which implies $\mu(A)\geqslant\mu(\bar{A})\geqslant\mu(\bar{\bar{A}})=\mu(A)$ for all $A\subseteq\T$ measurable.

To finish the proof we will see the equivalence $(i)\Leftrightarrow(v)$. Suppose that $\rho(q)H^2=qH^2$. By the unicity part of Theorem \ref{thm:beurling} there exists $c\in\T$ such that $q(z)=c\,\rho(q)(z)$ for almost all $z\in\T$. Then 
\begin{equation*}
\widehat{q}(n)=\scal{c\rho(q),S^n1}=c\scal{\rho(q),S^n1}=c\,\overline{\widehat{q}(n)}
\end{equation*}
holds for all $n\geqslant0$. Let $c=e^{i\theta}$ for some $\theta\in[0,2\pi)$. Then multiplying both sides in the previous equation by $\widehat{q}(n)$ gives us $\widehat{q}(n)^2=e^{i\theta}\abs{\widehat{q}(n)}^2$ for all $n\geqslant0$, or equivalently, $\widehat{q}(n)=\pm e^{i\theta/2}\abs{\widehat{q}(n)}$. Therefore, if we call $\alpha = e^{-i\theta/2}$, we have $\alpha\,\widehat{q}(n)\in\R$ for all $n\geqslant0$.

Conversely, if $q$ is an inner function whose Fourier expansion is 
$q(z)=\alpha\sum_{n\geqslant0}\beta_nz^n$ for some $\alpha\in\T$ and $\{\beta_n\}_n\subseteq\R$ then $\rho(q)(z)=\overline{\alpha}\sum_{n\geqslant0}\beta_nz^n=\frac{\overline{\alpha}}{\alpha}q(z)$ and thus $qH^2=\rho(q)H^2$.
\end{proof}

\section*{Acknowledgments}
We thank the referee for the careful reading of our paper and for the valuable suggestions that helped improve the exposition.

This research was supported by grants: UBACyT 20020170100430BA, PICT 2018-3399 (ANPCyT), PICT 2019-03968 (ANPCyT) and CONICET PIP 11220110101018. F.N. was supported by the postdoctoral fellowship 10520200102714CO of CONICET.

\end{document}